\documentclass[12pt,leqno]{amsart}
\usepackage[a4paper,hmargin={2.5cm,2.5cm},vmargin={2.5cm,2.5cm}]{geometry}

\usepackage[utf8]{inputenc}
\usepackage[T1]{fontenc}
\usepackage{lmodern}
\usepackage{mathrsfs}
\usepackage[english]{babel}
\usepackage{csquotes}
\usepackage{cite}
\usepackage[shortlabels]{enumitem}
\usepackage{libertinus}
\usepackage{libertinust1math}
\usepackage{mathtools}
\mathtoolsset{showonlyrefs}
\usepackage{amssymb}

\usepackage{graphicx}
\usepackage{float}
\usepackage{xcolor}
\usepackage{tikz-cd}
\usepackage{pgfplots}
\usepgfplotslibrary{fillbetween}
\pgfplotsset{compat=1.18}
\usepackage{multicol}
\usepackage{dsfont}
\usepackage{hyperref}

\numberwithin{equation}{section}

\theoremstyle{plain}
\newtheorem{theorem}{Theorem}[section]
\newtheorem{lemma}[theorem]{Lemma}
\newtheorem{proposition}[theorem]{Proposition}
\newtheorem{corollary}[theorem]{Corollary}

\newtheorem{remark}[theorem]{Remark}
\newtheorem*{remark*}{Remark}

\title[Vectorial Sobolev Minimizers]{Sobolev Minimizers for Critical Vectorial
$p$-Laplacian Problems}
\author{Simone Mauro}
\address{Dipartimento di Matematica e Informatica, Universit\`a della
Calabria, Ponte Pietro Bucci, Cubo 31B, 87036 Arcavacata di Rende, Cosenza,
Italy}
\email{simone.mauro@unical.it}
\date{\today}

\begin{document}

\begin{abstract}
We study minimizers of the optimal Sobolev quotients for the
critical embeddings
$D^{1,p}(\mathbb R^N;\mathbb R^m)\hookrightarrow
L^{p^*}(\mathbb R^N;\mathbb R^m)$ and
$W^{1,p}(\Omega;\mathbb R^m)\hookrightarrow
L^{p^*}(\Omega;\mathbb R^m)$. In the whole space, we classify all
minimizers as
$\boldsymbol u=\boldsymbol c\,U_{\delta,x_0}$, where
$\boldsymbol c\neq\boldsymbol0$ and $U_{\delta,x_0}$ is an
Aubin--Talenti bubble.

On bounded domains and for $\lambda>0$, we equip
$W^{1,p}(\Omega;\mathbb R^m)$ with
\(
  \|\boldsymbol u\|_\lambda^p
  :=\int_\Omega
  \bigl(|D\boldsymbol u|^p+\lambda|\boldsymbol u|^p\bigr)
\).
Using scalar attainment results, we establish existence
and synchronization of vectorial minimizers.
At resonance, we consider the pure Neumann problem
\(
  -\boldsymbol\Delta_p\boldsymbol u
  =|\boldsymbol u|^{p^*-2}\boldsymbol u\ \text{in }\Omega.
\)
For the associated quotient, we prove existence and global
$C^1(\overline\Omega)$ regularity of minimizers.
Finally, we prove synchronization of radial solutions and
deduce symmetry breaking on balls.
\end{abstract}


\maketitle
\tableofcontents

\section{Introduction}

Let $1<p<N$, $p^*=Np/(N-p)$, and
$\boldsymbol u=(u^1,\ldots,u^m)\colon\mathbb R^N\to\mathbb R^m$. We write
\[
  |D\boldsymbol u|^2:=\sum_{i=1}^m|\nabla u^i|^2,
  \qquad
  \boldsymbol\Delta_p\boldsymbol u
  :=\mathbf{div}\bigl(|D\boldsymbol u|^{p-2}D\boldsymbol u\bigr).
\]
For existence, multiplicity, and regularity results for quasilinear elliptic
systems in the subcritical setting, we refer to
\cite{CaninoMauroVectorial,MauroMultiplicityRegularity}. Against this
background, we first consider the critical vectorial problem
\[
  -\boldsymbol\Delta_p\boldsymbol u
  =|\boldsymbol u|^{p^*-2}\boldsymbol u
  \qquad\text{in }\mathbb R^N,
  \tag{$\mathcal P_\infty$}\label{Pinf}
\]
associated with the vectorial Sobolev quotient
\[
  \mathcal Q_\infty(\boldsymbol u)
  :=\frac{\displaystyle\int_{\mathbb R^N}|D\boldsymbol u|^p}
  {\displaystyle
   \left(\int_{\mathbb R^N}|\boldsymbol u|^{p^*}\right)^{p/p^*}},
  \qquad
  S_m:=\inf_{\boldsymbol u\in
  D^{1,p}(\mathbb R^N;\mathbb R^m)\setminus\{\boldsymbol0\}}
  \mathcal Q_\infty(\boldsymbol u).
\]
By homogeneity, every minimizer of \(S_m\), after multiplication by a
suitable positive constant, is a weak solution of \eqref{Pinf}.

Thus $S_m$ is the optimal constant in the critical Sobolev inequality
\[
  S_m\left(\int_{\mathbb R^N}|\boldsymbol u|^{p^*}\right)^{p/p^*}
  \leq\int_{\mathbb R^N}|D\boldsymbol u|^p,
  \qquad
  \boldsymbol u\in D^{1,p}(\mathbb R^N;\mathbb R^m).
\]
The constant $S_m$ is also optimal in
\[
  S_m\left(\int_\Omega|\boldsymbol u|^{p^*}\right)^{p/p^*}
  \leq\int_\Omega|D\boldsymbol u|^p,
  \qquad
  \boldsymbol u\in W_0^{1,p}(\Omega;\mathbb R^m),
\]
for every nonempty open set $\Omega\subset\mathbb R^N$. Indeed, extension
by zero gives one inequality, while the reverse one follows by placing
suitably rescaled compactly supported competitors inside a ball contained
in $\Omega$.

The scalar optimal constant, obtained by taking $m=1$, will be denoted by
$S$. It is attained by the explicit functions
\begin{equation}\label{eq:bubble}
  U_{\delta,x_0}(x)
  :=\left(
  \frac{N^{1/p}\left(\dfrac{N-p}{p-1}\right)^{(p-1)/p}
  \delta^{1/(p-1)}}
  {\delta^{p/(p-1)}+|x-x_0|^{p/(p-1)}}
  \right)^{(N-p)/p},
\end{equation}
where $\delta>0$ and $x_0\in\mathbb R^N$. These functions solve
$-\Delta_pU=U^{p^*-1}$. For $1<p<2$, the classification of all positive
solutions in $D^{1,p}(\mathbb R^N)$ was first proved in the range
$p^*\geq2$ in \cite{DamascelliMerchanMontoroSciunzi2014} and then in the
full range in \cite{Vetois2016}. The semilinear case $p=2$ was established
in \cite{CaffarelliGidasSpruck1989,ChenLi1991}, while the case $2<p<N$ was
completed in \cite{Sciunzi2016}.

Our first result gives a complete classification of the minimizers of the
optimal vectorial Sobolev constant $S_m$ defined above. It may also be viewed
as the critical counterpart of the classification principle for
least-energy vectorial solutions developed in \cite{CaninoMauroVectorial}.

\begin{theorem}\label{thm:whole-space}
The constant $S_m$ is attained and $S_m=S$. Moreover, every minimizer
$\boldsymbol u$ of $S_m$ has the form
\[
  \boldsymbol u
  =\bigl(c^1U_{\delta,x_0},\ldots,c^mU_{\delta,x_0}\bigr),
\]
for some $\boldsymbol c=(c^1,\ldots,c^m)\in\mathbb R^m\setminus
\{\boldsymbol0\}$, $\delta>0$, and $x_0\in\mathbb R^N$, where
$U_{\delta,x_0}$ is the Aubin--Talenti bubble defined in \eqref{eq:bubble}. Conversely,
every map of this form is a minimizer.
\end{theorem}


After the whole-space classification, we turn to the corresponding
perturbed problem on a bounded, connected
domain $\Omega\subset\mathbb R^N$ of class $C^2$, subject to Neumann boundary
conditions:

\begin{equation}\tag{$\mathcal P_\lambda$}\label{Pl}
  \begin{cases}
    -\boldsymbol\Delta_p\boldsymbol u
      +\lambda|\boldsymbol u|^{p-2}\boldsymbol u
      =|\boldsymbol u|^{p^*-2}\boldsymbol u & \text{in }\Omega, \\
    |D\boldsymbol u|^{p-2}D\boldsymbol u\,\nu=0
      & \text{on }\partial\Omega.
  \end{cases}
\end{equation}
Here $D\boldsymbol u$
is the $m\times N$ Jacobian matrix of $\boldsymbol u$, while
$\nu\in\mathbb R^N$ denotes the outward unit normal to $\partial\Omega$.
Thus $D\boldsymbol u\,\nu$ is the matrix--vector product in $\mathbb R^m$,
namely
\[
  D\boldsymbol u\,\nu
  =\bigl(\nabla u^1\cdot\nu,\ldots,\nabla u^m\cdot\nu\bigr)
  =\bigl(\partial_\nu u^1,\ldots,\partial_\nu u^m\bigr).
\]
Consequently, the boundary condition is understood componentwise as
$|D\boldsymbol u|^{p-2}\partial_\nu u^i=0$ on $\partial\Omega$ for every
$i=1,\ldots,m$.

For $\lambda>0$, define
\[
  \mathcal Q_\lambda(\boldsymbol u)
  :=\frac{\displaystyle\int_\Omega
  \bigl(|D\boldsymbol u|^p+\lambda|\boldsymbol u|^p\bigr)}
  {\displaystyle
  \left(\int_\Omega|\boldsymbol u|^{p^*}\right)^{p/p^*}},
  \qquad
  S_{\lambda,m}
  :=\inf_{\boldsymbol u\in W^{1,p}(\Omega;\mathbb R^m)
  \setminus\{\boldsymbol0\}}\mathcal Q_\lambda(\boldsymbol u).
\]
Every minimizer of \(S_{\lambda,m}\), after multiplication by a
suitable positive constant, is a weak solution of \eqref{Pl}.
We denote its scalar counterpart by $S_\lambda$. In this notation,
$S_{\lambda,m}$ is the optimal constant in the critical Sobolev inequality
\[
  S_{\lambda,m}
  \left(\int_\Omega|\boldsymbol u|^{p^*}\right)^{p/p^*}
  \leq\int_\Omega
  \bigl(|D\boldsymbol u|^p+\lambda|\boldsymbol u|^p\bigr),
  \qquad
  \boldsymbol u\in W^{1,p}(\Omega;\mathbb R^m).
\]
Equivalently, it is the optimal constant for the critical embedding of
$W^{1,p}(\Omega;\mathbb R^m)$, endowed with the norm induced by the
right-hand side, into $L^{p^*}(\Omega;\mathbb R^m)$. This embedding is
continuous but not compact in general at the critical exponent.

The scalar case $m=1$ has been extensively studied in connection with
critical Neumann problems. In the case $p=2$, positive solutions and the
associated sharp Sobolev-type inequalities were investigated in
\cite{AdimurthiYadava1994}, while the existence and nonexistence of
least-energy solutions for critical quasilinear equations with a lower-order
perturbation were studied in \cite{Ferrero2006}. 

{

More recently, existence of scalar minimizers was established
in \cite{CHS2026} for critical $p$-Laplacian problems with mixed
boundary conditions. In particular, \cite[Theorem~1.4]{CHS2026}
includes the attainment of $S_\lambda$ for every $\lambda>0$
under the dimensional assumption $N\ge 2p-1$.
Beyond scalar attainment, we characterize the role of constant
minimizers in the positive-mass problem and study the critical
pure Neumann problem under a nonlinear centering constraint.
Section~2 collects the scalar results needed below, including
the boundary bubble estimates used in the resonant case.
Building on the positive-mass scalar theory, our second main
theorem establishes attainment in the vectorial setting and
characterizes all vectorial minimizers.

}


\begin{theorem}\label{thm:bounded-vector}
{Assume that \(N\geq2p-1\), and let \(\lambda>0\).}
Then $S_\lambda$ is attained by a positive scalar minimizer. Moreover,
$S_{\lambda,m}$ is attained and
$S_{\lambda,m}=S_\lambda$. Every vectorial minimizer is
synchronized: there exist $\boldsymbol c\in\mathbb S^{m-1}$ and a positive
scalar minimizer $\omega\in C^1(\overline\Omega)$ of $S_\lambda$ such that
\(
  \boldsymbol u=\boldsymbol c\,\omega.
\)
After rescaling, $\omega$ satisfies
\[
  \begin{cases}
    -\Delta_p\omega+\lambda\omega^{p-1}=\omega^{p^*-1}
      &\text{in }\Omega,\\
    |\nabla\omega|^{p-2}\partial_\nu\omega=0
      &\text{on }\partial\Omega.
  \end{cases}
\]
Finally, the following alternatives describe the constant minimizers. If
\(p>2\), every minimizer of \(S_{\lambda,m}\) is nonconstant. If
\(1<p\leq2\), define
\[
  \lambda_0
  :=\sup\left\{\mu>0:
  \int_\Omega|\nabla v|^p+\mu\int_\Omega|v|^p
  \geq\mu|\Omega|^{p/N}
  \quad\text{for every }v\in W^{1,p}(\Omega):\ \int_\Omega|v|^{p^*}=1\right\}.
\]
Then every minimizer of \(S_{\lambda,m}\) is constant for
\(0<\lambda<\lambda_0\), constant maps remain minimizers at
\(\lambda=\lambda_0\), and no nonzero constant map is a minimizer for
\(\lambda>\lambda_0\).
\end{theorem}

The assertions concerning constant minimizers follow from
Lemma~\ref{lem:scalar-constants-not-minimizers},
Propositions~\ref{prop:small-lambda-constants} and
\ref{prop:optimal-constant-threshold}, and the synchronization
statement above.

The regularity assertion in Theorem~\ref{thm:bounded-vector} follows from
the global elliptic regularity theory for the scalar $p$-Laplacian, since
the vectorial minimizers are synchronized. More precisely, the
$L^\infty$ estimate follows from
\cite[Theorem~2.2]{PucciServadei2008}, while the boundary regularity theory
in \cite{Lieberman1988} yields
$\omega\in C^{1,\beta}(\overline\Omega)$ for some $\beta\in(0,1)$.

\begin{remark}
The same synchronization result holds under homogeneous Dirichlet boundary
conditions. In the classical Brezis--Nirenberg setting,
let $\lambda_{1,p}$ be the first Dirichlet eigenvalue of the $p$-Laplacian
and, for $0<\mu<\lambda_{1,p}$, define
\[
  S^{D}_{\mu,m}:=
  \inf_{\boldsymbol u\in W_0^{1,p}(\Omega;\mathbb R^m)
  \setminus\{\boldsymbol0\}}
  \frac{\displaystyle\int_\Omega
  \bigl(|D\boldsymbol u|^p-\mu|\boldsymbol u|^p\bigr)}
  {\displaystyle
  \left(\int_\Omega|\boldsymbol u|^{p^*}\right)^{p/p^*}}.
\]
Whenever the corresponding scalar constant $S^D_\mu$ is attained by a
positive function $\omega$, the argument of Theorem~\ref{thm:bounded-vector}
gives
\[
  S^D_{\mu,m}=S^D_\mu,
  \qquad
  \boldsymbol u=\boldsymbol c\,\omega,
  \quad \boldsymbol c\in\mathbb S^{m-1},
\]
for every vectorial minimizer. The required scalar existence results for the
critical Dirichlet problem are available, in the relevant ranges of the
dimension and of the parameter, in
\cite{ArioliGazzola1998,GarciaAzoreroPeral1987}. Thus the corresponding
Dirichlet version follows by combining those scalar results with the modulus
inequality and its equality case used here.

Related Brezis--Nirenberg-type existence and nonexistence
results for critical Dirichlet problems involving the
anisotropic $p$-Laplacian have recently been established
in \cite{BiagiEspositoRoncoroniVecchi2025}.
\end{remark}

Having treated the positive-mass Neumann problem and briefly discussed its
Dirichlet counterpart, we finally turn to the resonant Neumann case
$\lambda=0$:
\begin{equation}\tag{$\mathcal P_0$}\label{P0}
  \begin{cases}
    -\boldsymbol\Delta_p\boldsymbol u
      =|\boldsymbol u|^{p^*-2}\boldsymbol u
      &\text{in }\Omega,\\
    |D\boldsymbol u|^{p-2}D\boldsymbol u\,\nu=0
      &\text{on }\partial\Omega.
  \end{cases}
\end{equation}
Every weak solution of \eqref{P0} satisfies the compatibility condition
\[
  \int_\Omega
  |\boldsymbol u|^{p^*-2}\boldsymbol u=\boldsymbol0.
\]
Indeed, this identity follows by testing the equation with constant vectors
and using the homogeneous Neumann boundary condition. In particular, each
component $u^i$ which is not identically zero must change sign in $\Omega$.

We therefore introduce the centered constraint
\[
  \mathcal A_m:=\left\{
  \boldsymbol u\in W^{1,p}(\Omega;\mathbb R^m)\setminus\{\boldsymbol0\}:
  \int_\Omega|\boldsymbol u|^{p^*-2}\boldsymbol u=\boldsymbol0
  \right\}
\]
and define the resonant Sobolev constant by
\[
  S_{0,m}:=\inf_{\boldsymbol u\in\mathcal A_m}
  \frac{\displaystyle\int_\Omega|D\boldsymbol u|^p}
  {\displaystyle
  \left(\int_\Omega|\boldsymbol u|^{p^*}\right)^{p/p^*}}.
\]
Its scalar counterpart is
\[
  S_0:=\inf_{u\in\mathcal A_1}
  \frac{\displaystyle\int_\Omega|\nabla u|^p}
  {\displaystyle
  \left(\int_\Omega|u|^{p^*}\right)^{p/p^*}},
  \qquad
  \mathcal A_1:=\left\{
  u\in W^{1,p}(\Omega)\setminus\{0\}:
  \int_\Omega|u|^{p^*-2}u=0
  \right\}.
\]
This constraint excludes every nonzero constant map.

In the scalar case $p=2$, this constraint already appears in
\cite{PariniWeth2015} and is also used for the critical pure Neumann
problem in \cite{SaldanaTavares2022}. A corresponding scalar construction
for the $p$-Laplacian, including the case $p\neq 2$, is used in
\cite{McCurdySaldanaSchiera2026}. We establish here its vectorial
counterpart at the critical exponent. The main difficulty is that, when
\(p^*<2\), the map defining \(\mathcal A_m\) is not of class \(C^1\), so
\(\mathcal A_m\) cannot be handled directly as a smooth constraint.

We overcome this issue by constructing a homeomorphism
\[
  \mathcal H\colon
  W_\perp^{1,p}(\Omega;\mathbb R^m)\setminus\{\boldsymbol0\}
  \longrightarrow\mathcal A_m,
  \qquad
  \mathcal H(\boldsymbol u)
  =\boldsymbol u+\boldsymbol d(\boldsymbol u),
\]
where
\[
  W_\perp^{1,p}(\Omega;\mathbb R^m)
  :=\left\{\boldsymbol u\in W^{1,p}(\Omega;\mathbb R^m):
  \int_\Omega\boldsymbol u=\boldsymbol0\right\}.
\]
The variational problem is then transferred to
\[
  \mathcal M
  :=\mathcal H^{-1}(\mathbb S_{p^*,m}\cap\mathcal A_m)
  =\left\{\boldsymbol u\in W_\perp^{1,p}(\Omega;\mathbb R^m):
  \int_\Omega
  |\boldsymbol u+\boldsymbol d(\boldsymbol u)|^{p^*}=1\right\},
\]
which is a \(C^1\) manifold of codimension one. Ekeland's principle can
therefore be applied on \(\mathcal M\).
Related reductions are well known in generalized
Nehari manifold methods and have also been used for competitive systems; see
\cite{SzulkinWeth2010,ClappSzulkin2019,CaninoMauroCompetitive}. The novelty
claimed here concerns instead the attainment and global \(C^1\)-regularity
result below in the critical resonant setting, which 
appears to be new even
for \(m=1, p\ne2\).

\begin{theorem}\label{thm:resonant-existence}
Assume that {\(N(N-2p+1)>p\)}.
Then $S_{0,m}$ is attained by some
$\boldsymbol u\in\mathcal A_m\cap
C^1(\overline\Omega;\mathbb R^m)$. After multiplication by a suitable
positive constant, $\boldsymbol u$ satisfies \eqref{P0} and is nontrivial.
\end{theorem}

Unlike the positive-mass case, synchronization at resonance remains an open
problem. Indeed, every nontrivial component of a solution of \eqref{P0}
changes sign and $|\boldsymbol u|\notin\mathcal A_1$, so the modulus argument
used in Theorem~\ref{thm:bounded-vector} does not apply.

\begin{remark}
For \(p=2\), the dimensional assumption in
Theorem~\ref{thm:resonant-existence} becomes
\(
  N(N-3)>2,
\)
which, since \(N\) is an integer, is equivalent to \(N>3\). This is
consistent with the dimensional range of the corresponding scalar results
in \cite{ComteKnaap1991}.
\end{remark}

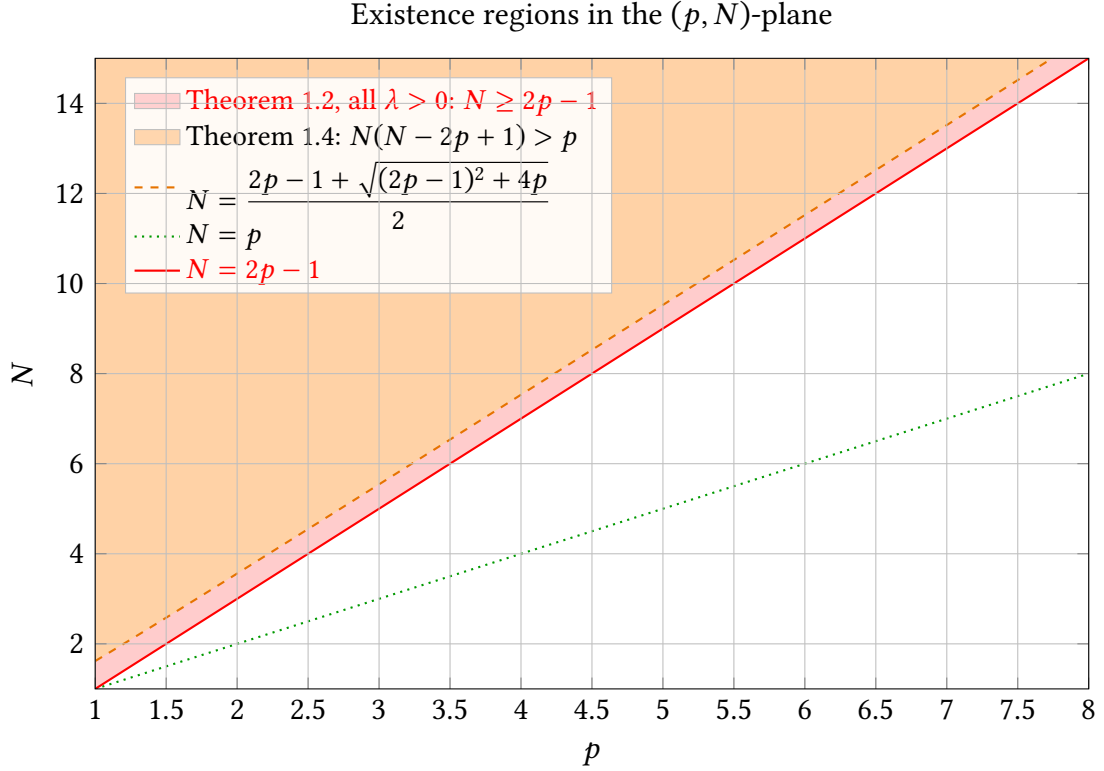
\begin{figure}[H]
\centering
\begin{tikzpicture}
\begin{axis}[
  width=0.92\textwidth,
  height=0.62\textwidth,
  xmin=1,
  xmax=8,
  ymin=1,
  ymax=15,
  xlabel={$p$},
  ylabel={$N$},
  title={Existence regions in the $(p,N)$-plane},
  grid=major,
  axis on top,
  legend pos=north west,
  legend cell align=left,
  legend style={font=\small,fill=white,fill opacity=0.9,
    text opacity=1,draw=gray!50},
]
  \path[name path=topblue]
    (axis cs:1,15)--(axis cs:8,15);
  \addplot[name path=bluefill,draw=none,forget plot,domain=1:8,samples=2]
    {2*x-1};
  \addplot[red!20,draw=none,forget plot]
    fill between[of=topblue and bluefill];

  \path[name path=toporange]
    (axis cs:1,15)--(axis cs:7.741935,15);
  \addplot[name path=orangefill,draw=none,forget plot,
    domain=1:7.741935,samples=150]
    {(2*x-1+sqrt((2*x-1)^2+4*x))/2};
  \addplot[orange!35,draw=none,forget plot]
    fill between[of=toporange and orangefill];

  \addlegendimage{area legend,fill=red!20,draw=none}
  \addlegendentry{\textcolor{red}{Theorem~\ref{thm:bounded-vector},
    all $\lambda>0$: $N\geq2p-1$}}
  \addlegendimage{area legend,fill=orange!35,draw=none}
  \addlegendentry{Theorem~\ref{thm:resonant-existence}:
    $N(N-2p+1)>p$}

  \addplot[orange!90!black,thick,dashed,domain=1:8,samples=150]
    {(2*x-1+sqrt((2*x-1)^2+4*x))/2};
  \addlegendentry{
    $N=\dfrac{2p-1+\sqrt{(2p-1)^2+4p}}{2}$}
  \addplot[green!60!black,thick,dotted,domain=1:8,samples=2]
    {x};
  \addlegendentry{$N=p$}
  \addplot[red,thick,domain=1:8,samples=2]
    {2*x-1};
  \addlegendentry{\textcolor{red}{$N=2p-1$}}
\end{axis}
\end{tikzpicture}
\end{figure}

Two conclusions are nevertheless available in radially symmetric domains.

\begin{theorem}\label{prop:radial-synchronization}
Let $\Omega$ be a ball or an annulus centered at the origin. Then every
radial solution of \eqref{P0} is synchronized.
\end{theorem}

On a ball, this leads to complete radial symmetry breaking.

\begin{corollary}\label{thm:ball-symmetry-breaking}
Let $\Omega=B_R$. Then every nontrivial solution of \eqref{P0} is
nonradial.
\end{corollary}

In particular, when $\Omega$ is a ball, the minimizers furnished by
Theorem~\ref{thm:resonant-existence} necessarily break radial symmetry.
If $p=2$ and $\Omega$ is an annulus, every minimizer of $S_{0,m}$ is
nonradial as well. Indeed, a radial vectorial minimizer would be synchronized
by Theorem~\ref{prop:radial-synchronization}, and its scalar factor would
attain $S_0$, contradicting
\cite[Proposition~1.4]{SaldanaTavares2022}.

\subsection{Notation}

The integer \(N\) denotes the dimension of the ambient domain, while
\(m\geq1\) is the number of components of the unknown map. We write
\[
  \boldsymbol u=(u^1,\ldots,u^m),
  \qquad
  |\boldsymbol u|^2=\sum_{i=1}^m|u^i|^2,
\]
and denote by \(D\boldsymbol u\) its \(m\times N\) Jacobian matrix. The
Frobenius norm and product are given by
\[
  |D\boldsymbol u|^2=\sum_{i=1}^m|\nabla u^i|^2,
  \qquad
  D\boldsymbol u:D\boldsymbol v
  =\sum_{i=1}^m\nabla u^i\cdot\nabla v^i.
\]
We use
\[
  D^{1,p}(\mathbb R^N;\mathbb R^m)
  :=\left\{
  \boldsymbol u\in L^{p^*}(\mathbb R^N;\mathbb R^m):
  D\boldsymbol u\in L^p(\mathbb R^N;\mathbb R^{m\times N})
  \right\},
\]
endowed with the norm \(\|D\boldsymbol u\|_{L^p}\); its scalar counterpart
is denoted by \(D^{1,p}(\mathbb R^N)\).

For convenience, we recall the constants introduced above. The constant
\(S_m\) is the optimal constant for the critical embedding
\(
  D^{1,p}(\mathbb R^N;\mathbb R^m)
  \hookrightarrow L^{p^*}(\mathbb R^N;\mathbb R^m),
\)
while \(S_{\lambda,m}\) is the optimal constant for the critical embedding
of \(W^{1,p}(\Omega;\mathbb R^m)\), endowed with the norm induced by
\[
  \int_\Omega
  \bigl(|D\boldsymbol u|^p+\lambda|\boldsymbol u|^p\bigr),
\]
into \(L^{p^*}(\Omega;\mathbb R^m)\). Finally, \(S_{0,m}\) is the optimal
constant for the same critical embedding at resonance, restricted to maps
satisfying the centering condition
\[
  \int_\Omega|\boldsymbol u|^{p^*-2}\boldsymbol u=\boldsymbol0,
\]
with energy \(\int_\Omega|D\boldsymbol u|^p\). Their scalar counterparts are
\[
  S:=S_1,
  \qquad
  S_\lambda:=S_{\lambda,1},
  \qquad
  S_0:=S_{0,1}.
\]
For \(1\leq q\leq p^*\), we also set
\[
  \mathbb S_{q,m}
  :=\left\{
  \boldsymbol u\in W^{1,p}(\Omega;\mathbb R^m):
  \int_\Omega|\boldsymbol u|^q=1
  \right\}.
\]

\section{The scalar problem on bounded domains}
\label{sec:scalar-bounded}

{Throughout this section, we assume that \(N\geq2p-1\).}

In this section we collect the scalar results concerning $S_\lambda$: the
strict inequality below the boundary concentration threshold, the resulting
compactness of minimizing sequences, and a criterion excluding constant
minimizers.

\begin{lemma}[Cherrier's inequality, \cite{Cherrier1984}]\label{lem:cherrier}
For every $\eta>0$, there exists a constant $C_\eta>0$ such that
\begin{equation}\label{eq:cherrier}
  \left(\int_\Omega|u|^{p^*}\right)^{p/p^*}
  \leq
  \left(\frac{2^{p/N}}{S}+\eta\right)
  \int_\Omega|\nabla u|^p
  +C_\eta\int_\Omega|u|^p
  \qquad\text{for every }u\in W^{1,p}(\Omega).
\end{equation}
\end{lemma}

{
\begin{lemma}\label{lem:ferrero-strict-bound}
{Assume that \(N\geq2p-1\).} For every \(\lambda>0\),
\[
  S_\lambda<\frac{S}{2^{p/N}}.
\]
\end{lemma}

\begin{proof}
{
Since \(\Omega\) is bounded and of class \(C^2\), one may choose a
supporting point \(x_0\in\partial\Omega\) at which the mean curvature is positive. After a translation and a rotation, \(x_0=0\),
and, near the origin,
\[
  \Omega=\{(x',x_N):x_N>h(x')\},
  \qquad
  h(x')=\frac12\sum_{j=1}^{N-1}\kappa_jx_j^2+o(|x'|^2),
\]
where
\(
  H(0):=(N-1)^{-1}\sum_{j=1}^{N-1}\kappa_j>0.
\)
We take the bubble \(U_\varepsilon\) defined in \eqref{eq:bubble} with
\(\delta=\varepsilon^{(p-1)/p}\) and center \(x_0=0\), namely
\(
  U_\varepsilon=U_{\varepsilon^{(p-1)/p},0}.
\)
Set
\[
  K_1:=\int_{\mathbb R^N}|\nabla U_\varepsilon|^p,
  \qquad
  K_2:=\int_{\mathbb R^N}U_\varepsilon^{p^*}.
\]
These quantities are independent of \(\varepsilon\), \(K_1=K_2\), and
\(
  S=K_1/K_2^{p/p^*}.
\)

Let \(\alpha:=p/p^*=(N-p)/N\). First suppose that \(N>2p-1\).
The estimates in \cite[Lemmas~1--4 and Section~5.1]{Ferrero2006}, after
the change of parameter \(\tau=\varepsilon^{(p-1)/p}\), yield positive
constants \(A_{N,p}\) and \(B_{N,p}\) such that
\[
  \int_\Omega|\nabla U_\varepsilon|^p
  =\frac{K_1}{2}-A_{N,p}H(0)\varepsilon^{(p-1)/p}
  +o\bigl(\varepsilon^{(p-1)/p}\bigr),
\]
\[
  \int_\Omega U_\varepsilon^{p^*}
  =\frac{K_2}{2}-B_{N,p}H(0)\varepsilon^{(p-1)/p}
  +o\bigl(\varepsilon^{(p-1)/p}\bigr),
  \qquad
  \int_\Omega U_\varepsilon^p
  =o\bigl(\varepsilon^{(p-1)/p}\bigr).
\]
Moreover, the computation in \cite[Section~5.1]{Ferrero2006} gives
\[
  \frac{A_{N,p}}{B_{N,p}}
  =\frac{N-p}{N}\frac{N+1}{N-2p+1}
  >\frac{N-p}{N}\frac{K_1}{K_2}.
\]
Consequently,
\[
\begin{aligned}
  \mathcal Q_\lambda(U_\varepsilon)
  &=\frac{S}{2^{p/N}}
  \left[1+2H(0)\left(
  \alpha\frac{B_{N,p}}{K_2}-\frac{A_{N,p}}{K_1}
  \right)\varepsilon^{(p-1)/p}
  +o\bigl(\varepsilon^{(p-1)/p}\bigr)\right]
  <\frac{S}{2^{p/N}}
\end{aligned}
\]
for all sufficiently small \(\varepsilon>0\).

If \(N=2p-1\), the borderline estimates in
\cite[Lemmas~1--4 and Section~5.2]{Ferrero2006} instead give a constant
\(A_p>0\) such that
\[
  \int_\Omega|\nabla U_\varepsilon|^p
  =\frac{K_1}{2}
  -A_pH(0)\varepsilon^{(p-1)/p}|\log\varepsilon|
  +o\bigl(\varepsilon^{(p-1)/p}|\log\varepsilon|\bigr),
\]
whereas
\[
  \int_\Omega U_\varepsilon^{p^*}
  =\frac{K_2}{2}+O\bigl(\varepsilon^{(p-1)/p}\bigr),
  \qquad
  \int_\Omega U_\varepsilon^p
  =O\bigl(\varepsilon^{(p-1)/p}\bigr).
\]
It follows that
\[
  \mathcal Q_\lambda(U_\varepsilon)
  =\frac{S}{2^{p/N}}
  \left[1-\frac{2A_pH(0)}{K_1}
  \varepsilon^{(p-1)/p}|\log\varepsilon|
  +o\bigl(\varepsilon^{(p-1)/p}|\log\varepsilon|\bigr)\right]
  <\frac{S}{2^{p/N}}
\]
for all sufficiently small \(\varepsilon>0\). In either case,
\(
  S_\lambda\leq\mathcal Q_\lambda(U_\varepsilon)<S/2^{p/N},
\)
which proves the claim.
}
\end{proof}
}

\begin{theorem}\label{thm:bounded-compactness}
Let $\lambda>0$ and assume that
\[
  S_\lambda<\frac{S}{2^{p/N}}.
\]
Then $S_\lambda$ is attained by a function
$\omega\in W^{1,p}(\Omega)\setminus\{0\}$ which is positive in $\Omega$.
\end{theorem}

\begin{proof}
Apply \cite[Lemma~2.1]{CHS2026} with
$\Gamma_0=\varnothing$, $\Gamma_1=\partial\Omega$,
$\alpha(x)\equiv\lambda$, and $\beta(x)\equiv0$.
The coercivity assumption is satisfied since $\lambda>0$,
and the corresponding Sobolev quotient is precisely
$\mathcal{Q}_\lambda$. Hence the strict inequality
$S_\lambda<S/2^{p/N}$ implies that $S_\lambda$ is attained.

Replacing a minimizer by its absolute value, we obtain
a nonnegative minimizer $\omega$, which we normalize by
$\int_\Omega \omega^{p^*}\,dx=1$. Its Euler--Lagrange equation is
\[
\begin{cases}
-\Delta_p\omega+\lambda\omega^{p-1}
    =S_\lambda\omega^{p^*-1} & \text{in }\Omega,\\
|\nabla\omega|^{p-2}\partial_\nu\omega=0
    & \text{on }\partial\Omega.
\end{cases}
\]
The $L^\infty$ estimate in
\cite[Theorem~2.2]{PucciServadei2008},
followed by the Harnack inequality
\cite{Trudinger1967}, yields $\omega>0$ in $\Omega$.
\end{proof}
\begin{lemma}\label{lem:scalar-constants-not-minimizers}
If $p>2$, then no nonzero constant function attains $S_\lambda$, for any
$\lambda>0$.
\end{lemma}

\begin{proof}
Let $p>2$, fix $c>0$, and choose a nonzero function
$\varphi\in C_c^\infty(\Omega)$ such that
\[
  \int_\Omega\varphi=0.
\]
For sufficiently small $|t|$, set $u_t:=c+t\varphi>0$. For every $q>1$,
Taylor's formula gives
\begin{equation}\label{eq:constant-expansion-Lq}
  \int_\Omega|u_t|^q
  =c^q|\Omega|
  +\frac{q(q-1)}{2}c^{q-2}t^2
  \int_\Omega\varphi^2+o(t^2).
\end{equation}
Moreover, since $p>2$,
\[
  \int_\Omega|\nabla u_t|^p
  =|t|^p\int_\Omega|\nabla\varphi|^p=o(t^2).
\]
Applying \eqref{eq:constant-expansion-Lq} with $q=p$ and $q=p^*$, we
obtain
\[
  \mathcal Q_\lambda(u_t)
  =\lambda|\Omega|^{p/N}
  \left[
  1-
  \frac{p(p^*-p)}{2c^2|\Omega|}t^2
  \int_\Omega\varphi^2+o(t^2)
  \right].
\]
Since $p^*>p$, it follows that
\[
  \mathcal Q_\lambda(u_t)
  <\mathcal Q_\lambda(c)=\lambda|\Omega|^{p/N}
\]
for every sufficiently small $t\neq0$. Thus no nonzero constant function
can attain $S_\lambda$.
\end{proof}

\begin{proposition}\label{prop:small-lambda-constants}
Assume that \(1<p\leq2\). Then there exists \(\overline\lambda>0\) such
that every scalar minimizer of \(S_\lambda\) is constant whenever
\(0<\lambda<\overline\lambda\).
\end{proposition}

\begin{proof}
We adapt the argument of \cite[Lemma~2.3]{AdimurthiYadavaLinNi1993}.
Let \(\widetilde w_\lambda\) be an arbitrary scalar minimizer, normalized by
\[
  \int_\Omega|\widetilde w_\lambda|^{p^*}=1.
\]
Its absolute value is a nonnegative minimizer and, by the \(L^\infty\)
estimate in \cite[Theorem~2.2]{PucciServadei2008} and the Harnack inequality
\cite{Trudinger1967}, it is positive in \(\Omega\). We may therefore work
with the positive minimizer \(w_\lambda:=|\widetilde w_\lambda|\), which
satisfies
\[
  -\Delta_pw_\lambda+\lambda w_\lambda^{p-1}
  =S_\lambda w_\lambda^{p^*-1}.
\]
Set
\[
  u_\lambda:=S_\lambda^{1/(p^*-p)}w_\lambda.
\]
Then
\begin{equation}\label{eq:small-lambda-rescaled-equation}
  -\Delta_pu_\lambda+\lambda u_\lambda^{p-1}
  =u_\lambda^{p^*-1}.
\end{equation}
Since constant functions give
\(S_\lambda\leq\lambda|\Omega|^{p/N}\), we have
\begin{equation}\label{eq:small-rescaled-minimizer}
  \|u_\lambda\|_{p^*}
  =S_\lambda^{1/(p^*-p)}\longrightarrow0
  \qquad\text{as }\lambda\to0.
\end{equation}

Write
\[
  c_\lambda:=\frac1{|\Omega|}\int_\Omega u_\lambda,
  \qquad
  \varphi_\lambda:=u_\lambda-c_\lambda.
\]
Thus \(\int_\Omega\varphi_\lambda=0\), and the Poincar\'e
inequality gives
\begin{equation}\label{eq:poincare-sobolev-small-lambda}
  \|\varphi_\lambda\|_{p^*}^p
  \leq C\int_\Omega|\nabla\varphi_\lambda|^p.
\end{equation}
Testing \eqref{eq:small-lambda-rescaled-equation} with
\(\varphi_\lambda\), subtracting the constant terms whose integrals vanish,
and using the monotonicity of \(t\mapsto t^{p-1}\), we obtain
\begin{equation}\label{eq:small-lambda-tested-equation}
  \int_\Omega|\nabla\varphi_\lambda|^p
  \leq\int_\Omega
  \bigl(u_\lambda^{p^*-1}-c_\lambda^{p^*-1}\bigr)
  \varphi_\lambda.
\end{equation}

Suppose first that \(p^*\geq2\). The inequality
\[
  |a^{p^*-1}-b^{p^*-1}|
  \leq C(a^{p^*-2}+b^{p^*-2})|a-b|,
  \qquad a,b\geq0,
\]
H\"older's inequality, and
\[
  |\Omega|^{1/p^*}c_\lambda
  \leq\|u_\lambda\|_{p^*},
  \qquad
  \|\varphi_\lambda\|_{p^*}
  \leq2\|u_\lambda\|_{p^*},
\]
give
\[
  \int_\Omega|\nabla\varphi_\lambda|^p
  \leq C\|u_\lambda\|_{p^*}^{p^*-2}
  \|\varphi_\lambda\|_{p^*}^2
  \leq C\|u_\lambda\|_{p^*}^{p^*-p}
  \|\varphi_\lambda\|_{p^*}^p,
\]
where the last inequality uses \(p\leq2\). Together with
\eqref{eq:poincare-sobolev-small-lambda}, this implies, if
\(\varphi_\lambda\not\equiv0\),
\[
  1\leq C\|u_\lambda\|_{p^*}^{p^*-p},
\]
contrary to \eqref{eq:small-rescaled-minimizer} for small \(\lambda\).

If \(p^*<2\), the concavity inequality
\[
  |a^{p^*-1}-b^{p^*-1}|
  \leq|a-b|^{p^*-1},
  \qquad a,b\geq0,
\]
gives instead
\[
  \int_\Omega|\nabla\varphi_\lambda|^p
  \leq\|\varphi_\lambda\|_{p^*}^{p^*}
  \leq C\|u_\lambda\|_{p^*}^{p^*-p}
  \|\varphi_\lambda\|_{p^*}^p.
\]
The same contradiction follows. Hence
\(\varphi_\lambda\equiv0\) for every sufficiently small \(\lambda\), so
\(w_\lambda\), and therefore \(\widetilde w_\lambda\), is constant. Since
the initial minimizer was arbitrary, the conclusion follows.
\end{proof}

\begin{proposition}\label{prop:optimal-constant-threshold}
Assume that $1<p\leq2$, and define
\begin{equation}\label{eq:optimal-constant-threshold}
  \lambda_0
  :=\sup\left\{\lambda>0:
  \int_\Omega|\nabla u|^p+\lambda\int_\Omega|u|^p
  \geq\lambda|\Omega|^{p/N}
  \quad\text{for every }u\in\mathbb S_{p^*,1}\right\}.
\end{equation}
Then
\(
  0<\lambda_0
  \leq\frac{S}{2^{p/N}|\Omega|^{p/N}}.
\)
Moreover,
\[
  S_\lambda=\lambda|\Omega|^{p/N}
  \quad\text{for }0<\lambda\leq\lambda_0,
  \qquad
  S_\lambda<\lambda|\Omega|^{p/N}
  \quad\text{for }\lambda>\lambda_0.
\]
Every minimizer of \(S_\lambda\) is constant for
\(0<\lambda<\lambda_0\); constant functions remain minimizers at
\(\lambda=\lambda_0\); and no nonzero constant function is a minimizer for
\(\lambda>\lambda_0\).
\end{proposition}

\begin{proof}
Let
\[
  \Lambda:=\left\{\lambda>0:
  \int_\Omega|\nabla u|^p+\lambda\int_\Omega|u|^p
  \geq\lambda|\Omega|^{p/N}
  \quad\text{for every }u\in\mathbb S_{p^*,1}\right\}.
\]
By Proposition~\ref{prop:small-lambda-constants}, there exists
\(\overline\lambda>0\) such that all minimizers of \(S_\lambda\) are
constant for \(0<\lambda<\overline\lambda\). Hence
\[
  S_\lambda=\lambda|\Omega|^{p/N}
  \qquad\text{for }0<\lambda<\overline\lambda,
\]
and therefore \((0,\overline\lambda)\subset\Lambda\). In particular,
\(\lambda_0>0\).

If \(\mu\in\Lambda\) and \(0<\lambda<\mu\), then, for every
\(u\in\mathbb S_{p^*,1}\),
\[
\begin{aligned}
  &\int_\Omega|\nabla u|^p+\lambda\int_\Omega|u|^p
  -\lambda|\Omega|^{p/N}\\
  &\quad=
  \left(
  \int_\Omega|\nabla u|^p+\mu\int_\Omega|u|^p
  -\mu|\Omega|^{p/N}
  \right)
  +(\mu-\lambda)
  \left(|\Omega|^{p/N}-\int_\Omega|u|^p\right)\geq0.
\end{aligned}
\]
Thus \(\Lambda\) is an interval.

Set
\(
  \lambda_*:=\frac{S}{2^{p/N}|\Omega|^{p/N}}.
\)
Lemma~\ref{lem:ferrero-strict-bound} gives
\(
  S_{\lambda_*}<\frac{S}{2^{p/N}}
  =\lambda_*|\Omega|^{p/N},
\)
so \(\lambda_*\notin\Lambda\). Hence
\(\lambda_0\leq\lambda_*\).

Choose \(\lambda_n\in\Lambda\) with \(\lambda_n\uparrow\lambda_0\).
 Therefore
\[
  S_\lambda=\lambda|\Omega|^{p/N}
  \qquad\text{for }0<\lambda\leq\lambda_0.
\]
{
Let $0<\lambda<\lambda_0$, and let $u$ be a minimizer of
$S_\lambda$, normalized by $\int_\Omega |u|^{p^*}=1$.
Since $S_{\lambda_0}=\lambda_0|\Omega|^{p/N}$, we have
\[
\lambda_0|\Omega|^{p/N}
\le \int_\Omega |\nabla u|^p+\lambda_0\int_\Omega |u|^p
= \lambda_0|\Omega|^{p/N}
  -\frac{\lambda_0-\lambda}{\lambda}
   \int_\Omega |\nabla u|^p.
\]
Thus $\nabla u=0$ almost everywhere in $\Omega$.
Since $\Omega$ is connected, $u$ is constant.

}

Finally, if \(\lambda>\lambda_0\), then \(\lambda\notin\Lambda\), and hence
some \(u\in\mathbb S_{p^*,1}\) satisfies
\[
  \int_\Omega|\nabla u|^p+
  \lambda\int_\Omega|u|^p
  <\lambda|\Omega|^{p/N}.
\]
Consequently,
\(S_\lambda<\lambda|\Omega|^{p/N}\), so no nonzero constant function is a
minimizer. At \(\lambda=\lambda_0\), constant functions remain minimizers.
\end{proof}

\section{Proofs of Theorems~\ref{thm:whole-space} and
\ref{thm:bounded-vector}}
\label{sec:proof-main}

\subsection{The whole-space problem}

We begin by comparing the scalar and vectorial Sobolev constants. Let
$\boldsymbol e_1,\ldots,\boldsymbol e_m$ denote the canonical basis of
$\mathbb R^m$.

\begin{proposition}\label{prop:scalar-to-vector}
Let $\omega\in D^{1,p}(\mathbb R^N)$ be a scalar minimizer for $S$. Then, for
every $i=1,\ldots,m$, the map
\[
  \boldsymbol u_i:=\omega\boldsymbol e_i
\]
is a minimizer for $S_m$. In particular, $S_m=S$.
\end{proposition}

\begin{proof}
Let $\boldsymbol u\in D^{1,p}(\mathbb R^N;\mathbb R^m)\setminus
\{\boldsymbol0\}$ and set $\rho:=|\boldsymbol u|$. Then
$\rho\in D^{1,p}(\mathbb R^N)$ and
\[
  |\nabla\rho|\leq|D\boldsymbol u|
  \qquad\text{a.e. in }\mathbb R^N.
\]
Since $\|\rho\|_{L^{p^*}}=\|\boldsymbol u\|_{L^{p^*}}$, it follows that
\[
  \mathcal Q_\infty(\boldsymbol u)
  \geq
  \frac{\displaystyle\int_{\mathbb R^N}|\nabla\rho|^p}
  {\displaystyle
  \left(\int_{\mathbb R^N}|\rho|^{p^*}\right)^{p/p^*}}
  \geq S.
\]
Taking the infimum gives $S_m\geq S$. On the other hand,
\[
  |\omega\boldsymbol e_i|=|\omega|,
  \qquad
  |D(\omega\boldsymbol e_i)|=|\nabla\omega|,
\]
and therefore
\[
  \mathcal Q_\infty(\omega\boldsymbol e_i)
  =\frac{\displaystyle\int_{\mathbb R^N}|\nabla\omega|^p}
  {\displaystyle
  \left(\int_{\mathbb R^N}|\omega|^{p^*}\right)^{p/p^*}}
  =S.
\]
Hence $S_m\leq S$, which proves $S_m=S$ and the assertion.
\end{proof}

We now classify all vectorial minimizers, following the synchronization
argument of \cite{CaninoMauroVectorial}.

\begin{proof}[Proof of Theorem~\ref{thm:whole-space}]
Proposition~\ref{prop:scalar-to-vector} shows that $S_m=S$ and that $S_m$ is
attained. Let
$\boldsymbol u\in D^{1,p}(\mathbb R^N;\mathbb R^m)$ be an arbitrary
minimizer for $S_m$ and set
\[
  \omega:=|\boldsymbol u|
  =\left(\sum_{i=1}^m|u^i|^2\right)^{1/2}.
\]
As in the proof of Proposition~\ref{prop:scalar-to-vector},
\[
  S=S_m=\mathcal Q_\infty(\boldsymbol u)
  \geq\mathcal Q_\infty(\omega)\geq S.
\]
Consequently, $\omega$ is a scalar minimizer for $S$ and
\[
  |\nabla\omega|=|D\boldsymbol u|
  \qquad\text{a.e. in }\mathbb R^N.
\]
Up to multiplication by a positive constant, $\omega$ is one of the
explicit functions $U_{\delta,x_0}$ introduced in \eqref{eq:bubble}; in particular,
$\omega>0$ in $\mathbb R^N$.

We recall Lagrange's identity: for $a_1,\ldots,a_m\in\mathbb R$ and
$\boldsymbol b_1,\ldots,\boldsymbol b_m\in\mathbb R^N$,
\[
  \left(\sum_{i=1}^m a_i^2\right)
  \left(\sum_{i=1}^m|\boldsymbol b_i|^2\right)
  -\left|\sum_{i=1}^m a_i\boldsymbol b_i\right|^2
  =\sum_{1\leq i<j\leq m}
  |a_i\boldsymbol b_j-a_j\boldsymbol b_i|^2.
\]
Applying this identity with $a_i=u^i$ and
$\boldsymbol b_i=\nabla u^i$, and observing that
\[
  \nabla\omega
  =\frac1\omega\sum_{i=1}^m u^i\nabla u^i,
\]
we have
\[
  \omega^2|\nabla\omega|^2
  =\omega^2|D\boldsymbol u|^2
  -\sum_{1\leq i<j\leq m}
  |u^i\nabla u^j-u^j\nabla u^i|^2.
\]
The equality $|\nabla\omega|=|D\boldsymbol u|$ therefore yields
\[
  u^i\nabla u^j=u^j\nabla u^i
  \qquad\text{a.e. in }\mathbb R^N,
  \quad i,j=1,\ldots,m.
\]
For each $i$, it follows that
\[
  \omega^2\nabla u^i
  =u^i\sum_{j=1}^m u^j\nabla u^j
  =u^i\omega\nabla\omega,
\]
and hence
\[
  \nabla\left(\frac{u^i}{\omega}\right)=0
  \qquad\text{a.e. in }\mathbb R^N.
\]
Thus $u^i=c^i\omega$ for some constants $c^i\in\mathbb R$. Since
$|\boldsymbol u|=\omega$, we have
\[
  \sum_{i=1}^m|c^i|^2=1.
\]
Therefore, $\boldsymbol u=\boldsymbol c\,\omega$ with
$\boldsymbol c\in\mathbb S^{m-1}$. Since
$\omega=\kappa U_{\delta,x_0}$ for some $\kappa>0$, the factor $\kappa$ can
be absorbed into the coefficients $c^i$. Hence
\[
  \boldsymbol u
  =\bigl(c^1U_{\delta,x_0},\ldots,c^mU_{\delta,x_0}\bigr)
\]
for a nonzero vector $\boldsymbol c\in\mathbb R^m$. Conversely,
Proposition~\ref{prop:scalar-to-vector} and rotational invariance show that
every map of this form is a minimizer. This concludes the proof.
\end{proof}

\subsection{The bounded-domain problem}

We now compare the scalar and vectorial constants, exactly as in the
whole-space case.

\begin{proposition}\label{prop:bounded-scalar-to-vector}
Let $\omega\in W^{1,p}(\Omega)$ be a scalar minimizer for $S_\lambda$. Then,
for every $i=1,\ldots,m$, the map $\omega\boldsymbol e_i$ is a minimizer for
$S_{\lambda,m}$. In particular,
\[
  S_{\lambda,m}=S_\lambda.
\]
\end{proposition}

\begin{proof}
For every
$\boldsymbol u\in W^{1,p}(\Omega;\mathbb R^m)\setminus\{\boldsymbol0\}$,
set $\rho:=|\boldsymbol u|$. Then
\[
  |\nabla\rho|\leq|D\boldsymbol u|,
  \qquad
  |\rho|=|\boldsymbol u|
  \quad\text{a.e. in }\Omega,
\]
and hence
\[
  \mathcal Q_\lambda(\boldsymbol u)
  \geq
  \frac{\displaystyle\int_\Omega
  \bigl(|\nabla\rho|^p+\lambda|\rho|^p\bigr)}
  {\displaystyle
  \left(\int_\Omega|\rho|^{p^*}\right)^{p/p^*}}
  \geq S_\lambda.
\]
Thus $S_{\lambda,m}\geq S_\lambda$. Conversely,
$\mathcal Q_\lambda(\omega\boldsymbol e_i)=S_\lambda$, which proves the
claim.
\end{proof}

\begin{proof}[Proof of Theorem~\ref{thm:bounded-vector}]
Lemma~\ref{lem:ferrero-strict-bound} and
Theorem~\ref{thm:bounded-compactness} provide a
positive scalar minimizer $\omega$ for $S_\lambda$. By
Proposition~\ref{prop:bounded-scalar-to-vector}, the vectorial constant is
attained and $S_{\lambda,m}=S_\lambda$.

Let $\boldsymbol u$ be any minimizer for $S_{\lambda,m}$ and set
$\rho:=|\boldsymbol u|$. The preceding comparison gives
\[
  S_\lambda=S_{\lambda,m}=\mathcal Q_\lambda(\boldsymbol u)
  \geq\mathcal Q_\lambda(\rho)\geq S_\lambda.
\]
Therefore, $\rho$ is a nonnegative  minimizer for $S_\lambda$ and
\[
  |\nabla\rho|=|D\boldsymbol u|
  \qquad\text{a.e. in }\Omega.
\]
The $L^\infty$ estimate in \cite[Theorem~2.2]{PucciServadei2008}, followed
by the Harnack inequality \cite{Trudinger1967}, gives
$\rho>0$ in $\Omega$. Repeating
the Lagrange-identity argument used in the proof of
Theorem~\ref{thm:whole-space}, we obtain
\[
  \nabla\left(\frac{u^i}{\rho}\right)=0,
  \qquad i=1,\ldots,m.
\]
Since $\Omega$ is connected, there exists
$\boldsymbol c=(c^1,\ldots,c^m)\in\mathbb S^{m-1}$ such that
\[
  \boldsymbol u=\boldsymbol c\,\rho.
\]
Conversely, every map $\boldsymbol c\,\omega$, with
$\boldsymbol c\in\mathbb S^{m-1}$ and $\omega$ a positive scalar minimizer,
attains $S_{\lambda,m}$.

After multiplication by a suitable positive constant, the scalar minimizer
$\rho$ satisfies
\[
  \begin{cases}
    -\Delta_p\rho+\lambda\rho^{p-1}=\rho^{p^*-1}
      &\text{in }\Omega,\\
    |\nabla\rho|^{p-2}\partial_\nu\rho=0
      &\text{on }\partial\Omega.
  \end{cases}
\]

It remains to prove the assertions concerning constant minimizers. If
\(p>2\), Lemma~\ref{lem:scalar-constants-not-minimizers} shows that \(\rho\)
cannot be constant. Since \(\boldsymbol u=\boldsymbol c\,\rho\), every
vectorial minimizer is therefore nonconstant.

Assume now that \(1<p\leq2\). Proposition~\ref{prop:small-lambda-constants}
and Proposition~\ref{prop:optimal-constant-threshold} show that the number
\(\lambda_0\) appearing in the statement is positive and has the following
properties. If \(0<\lambda<\lambda_0\), every scalar minimizer is constant;
hence the synchronization formula
\(\boldsymbol u=\boldsymbol c\,\rho\) shows that every vectorial minimizer is
constant. At \(\lambda=\lambda_0\), scalar constant functions are
minimizers, and Proposition~\ref{prop:bounded-scalar-to-vector} yields
constant vectorial minimizers. Finally, if \(\lambda>\lambda_0\), no
nonzero constant function minimizes \(S_\lambda\). If a vectorial minimizer
were constant, then its modulus \(\rho\) would be a nonzero constant scalar
minimizer by the equality chain above, a contradiction. Thus no nonzero
constant vectorial map is a minimizer in this range.

This concludes the proof.
\end{proof}

\section{Proof of Theorem~\ref{thm:resonant-existence}}
\label{sec:resonant}

We begin with the centering lemma specific to the kernel of the Neumann
$p$-Laplacian. 
\begin{lemma}\label{lem:centering}
For every $\boldsymbol u\in W^{1,p}(\Omega;\mathbb R^m)$, there exists a
unique vector $\boldsymbol d=\boldsymbol d(\boldsymbol u)\in\mathbb R^m$ such
that
\begin{equation}\label{eq:centering-map}
  \int_\Omega
  |\boldsymbol u+\boldsymbol d(\boldsymbol u)|^{p^*-2}
  \bigl(\boldsymbol u+\boldsymbol d(\boldsymbol u)\bigr)
  =\boldsymbol0.
\end{equation}
Moreover, the map
\(
  L^{p^*}(\Omega;\mathbb R^m)\ni\boldsymbol u
  \longmapsto\boldsymbol d(\boldsymbol u)\in\mathbb R^m
\)
is continuous and homogeneous, namely
\begin{equation}\label{eq:centering-map-homogeneity}
  \boldsymbol d(t\boldsymbol u)=t\boldsymbol d(\boldsymbol u)
  \qquad\text{for every }t\in\mathbb R.
\end{equation}
\end{lemma}

\begin{proof}
For a fixed $\boldsymbol u\in W^{1,p}(\Omega;\mathbb R^m)$, define
\[
  F(\boldsymbol d):=\int_\Omega
  |\boldsymbol u+\boldsymbol d|^{p^*},
  \qquad \boldsymbol d\in\mathbb R^m.
\]
Since $p^*>1$, the map
$\boldsymbol z\mapsto|\boldsymbol z|^{p^*}$ is strictly
convex. Hence $F$ is strictly convex. Moreover,
$F(\boldsymbol d)\to+\infty$ as $|\boldsymbol d|\to\infty$, so $F$ admits a
unique minimizer $\boldsymbol d(\boldsymbol u)$. For every
$\boldsymbol a\in\mathbb R^m$, the Euler equation at this minimizer is
\[
  0=\langle F'(\boldsymbol d(\boldsymbol u)),\boldsymbol a\rangle
  =p^*\,
  \int_\Omega
  |\boldsymbol u+\boldsymbol d(\boldsymbol u)|^{p^*-2}
  (\boldsymbol u+\boldsymbol d(\boldsymbol u))\cdot\boldsymbol a.
\]
Since $\boldsymbol a$ is arbitrary, \eqref{eq:centering-map} follows.

To prove \eqref{eq:centering-map-homogeneity}, let \(t\neq0\). By
\eqref{eq:centering-map},
\[
  \int_\Omega
  |t(\boldsymbol u+\boldsymbol d(\boldsymbol u))|^{p^*-2}
  t(\boldsymbol u+\boldsymbol d(\boldsymbol u))=\boldsymbol0.
\]
Thus, uniqueness gives
\(\boldsymbol d(t\boldsymbol u)=t\boldsymbol d(\boldsymbol u)\). For
\(t=0\), the conclusion follows from
\(\boldsymbol d(\boldsymbol0)=\boldsymbol0\).

It remains to prove continuity. Let
\(\boldsymbol u_n\to\boldsymbol u\) in
\(L^{p^*}(\Omega;\mathbb R^m)\), and set
\(\boldsymbol d_n:=\boldsymbol d(\boldsymbol u_n)\). By the minimizing
property,
\[
  \|\boldsymbol u_n+\boldsymbol d_n\|_{L^{p^*}}
  \leq\|\boldsymbol u_n\|_{L^{p^*}}.
\]
Hence the triangle inequality gives
\[
  |\Omega|^{1/p^*}|\boldsymbol d_n|
  \leq
  \|\boldsymbol u_n+\boldsymbol d_n\|_{L^{p^*}}
  +\|\boldsymbol u_n\|_{L^{p^*}}
  \leq2\|\boldsymbol u_n\|_{L^{p^*}},
\]
so \((\boldsymbol d_n)\) is bounded in \(\mathbb R^m\). Let
\(\boldsymbol d_{n_k}\to\overline{\boldsymbol d}\) be a convergent
subsequence. Then
\[
  \boldsymbol u_{n_k}+\boldsymbol d_{n_k}
  \longrightarrow\boldsymbol u+\overline{\boldsymbol d}
  \quad\text{in }L^{p^*}(\Omega;\mathbb R^m).
\]
By continuity of the duality map from \(L^{p^*}\) to
\(L^{(p^*)'}\), we may pass to the limit in
\eqref{eq:centering-map}, written for \(\boldsymbol u_n\), and obtain
\[
  \int_\Omega
  |\boldsymbol u+\overline{\boldsymbol d}|^{p^*-2}
  (\boldsymbol u+\overline{\boldsymbol d})=\boldsymbol0.
\]
Uniqueness yields
\(\overline{\boldsymbol d}=\boldsymbol d(\boldsymbol u)\).
Since every convergent subsequence of \((\boldsymbol d_n)\) has the same
limit, the whole sequence converges to \(\boldsymbol d(\boldsymbol u)\).
\end{proof}

We now use Lemma~\ref{lem:centering} to eliminate the centering constraint.
Let
\[
  W_\perp^{1,p}(\Omega;\mathbb R^m)
  :=\left\{
  \boldsymbol u\in W^{1,p}(\Omega;\mathbb R^m):
  \int_\Omega\boldsymbol u=\boldsymbol0
  \right\},
\]
endowed with the norm
\[
  \|\boldsymbol u\|_\perp
  :=\left(\int_\Omega|D\boldsymbol u|^p\right)^{1/p},
\]
which is equivalent to the \(W^{1,p}\)-norm on this space by the
Poincar\'e inequality.

\begin{lemma}\label{lem:centering-homeomorphism}
The map
\[
  \mathcal H\colon
  W_\perp^{1,p}(\Omega;\mathbb R^m)\setminus\{\boldsymbol0\}
  \longrightarrow\mathcal A_m,
  \qquad
  \boldsymbol v=\mathcal H(\boldsymbol u)
  :=\boldsymbol u+\boldsymbol d(\boldsymbol u)
\]
is a homeomorphism. Its inverse is
\begin{equation}\label{eq:centering-homeomorphism-inverse}
  \mathcal H^{-1}(\boldsymbol v)
  =\boldsymbol v-\frac1{|\Omega|}\int_\Omega\boldsymbol v.
\end{equation}
\end{lemma}

\begin{proof}
By \eqref{eq:centering-map}, the map \(\mathcal H\) takes values in
\(\mathcal A_m\). Let \(\boldsymbol v\in\mathcal A_m\) and set
\(
  \boldsymbol u
  =\boldsymbol v-\frac1{|\Omega|}\int_\Omega\boldsymbol v
\). Then
\(\boldsymbol u\in W_\perp^{1,p}(\Omega;\mathbb R^m)\). Moreover,
\[
  \boldsymbol u+\frac1{|\Omega|}\int_\Omega\boldsymbol v
  =\boldsymbol v,
\]
and  Lemma~\ref{lem:centering} gives
\[
  \boldsymbol d(\boldsymbol u)
  =\frac1{|\Omega|}\int_\Omega\boldsymbol v,
  \qquad
  \mathcal H(\boldsymbol u)=\boldsymbol v.
\]
This proves that \(\mathcal H\) is bijective and that its
inverse is given by
\eqref{eq:centering-homeomorphism-inverse}.
The continuity of \(\mathcal H\) follows from
Lemma~\ref{lem:centering}, while the continuity of its inverse is immediate.
\end{proof}

Consequently,
\begin{equation}\label{eq:resonant-reduced-quotient}
  S_{0,m}
  =\inf_{\boldsymbol u\in
  W_\perp^{1,p}(\Omega;\mathbb R^m)\setminus\{\boldsymbol0\}}
  \frac{\displaystyle\int_\Omega|D\boldsymbol u|^p}
  {\displaystyle
  \left(
  \int_\Omega
  |\boldsymbol u+\boldsymbol d(\boldsymbol u)|^{p^*}
  \right)^{p/p^*}}.
\end{equation}
By \eqref{eq:centering-map-homogeneity}, the quotient in
\eqref{eq:resonant-reduced-quotient} is homogeneous.
Therefore, with the notation introduced above, we have
\[
  \mathcal H^{-1}(\mathbb S_{p^*,m}\cap\mathcal A_m)
  =\left\{
  \boldsymbol u\in W_\perp^{1,p}(\Omega;\mathbb R^m):
  \int_\Omega
  |\boldsymbol u+\boldsymbol d(\boldsymbol u)|^{p^*}=1
  \right\}.
\]
Thus the equivalent minimization problems are
\begin{equation}\label{eq:resonant-sphere-reduction}
  S_{0,m}
  =\inf_{\boldsymbol v\in\mathbb S_{p^*,m}\cap\mathcal A_m}
  \int_\Omega|D\boldsymbol v|^p
  =\inf_{\boldsymbol u\in
  \mathcal H^{-1}(\mathbb S_{p^*,m}\cap\mathcal A_m)}
  \int_\Omega|D\boldsymbol u|^p.
\end{equation}

\begin{lemma}\label{lem:centered-functional-c1}
The functional
\[
  G\colon W_\perp^{1,p}(\Omega;\mathbb R^m)\longrightarrow\mathbb R,
  \qquad
  G(\boldsymbol u)
  :=\frac1{p^*}\int_\Omega
  |\boldsymbol u+\boldsymbol d(\boldsymbol u)|^{p^*},
\]
is of class \(C^1\). Moreover,
\begin{equation}\label{eq:derivative-centered-functional}
  G'(\boldsymbol u)[\boldsymbol\varphi]
  =\int_\Omega
  |\boldsymbol u+\boldsymbol d(\boldsymbol u)|^{p^*-2}
  (\boldsymbol u+\boldsymbol d(\boldsymbol u))\cdot
  \boldsymbol\varphi
\end{equation}
for every \(\boldsymbol u,\boldsymbol\varphi\in W_\perp^{1,p}\).
\end{lemma}

\begin{proof}
By Lemma~\ref{lem:centering}, for every
\(\boldsymbol u\in W_\perp^{1,p}\),
\[
  G(\boldsymbol u)
  =\min_{\boldsymbol d\in\mathbb R^m}
  \frac1{p^*}\int_\Omega|\boldsymbol u+\boldsymbol d|^{p^*},
\]
and the minimum is attained only at
\(\boldsymbol d=\boldsymbol d(\boldsymbol u)\).
Consider the \(C^1\) functional
\[
  \mathcal P\colon L^{p^*}(\Omega;\mathbb R^m)\longrightarrow\mathbb R,
  \qquad
  \mathcal P(\boldsymbol z)
  :=\frac1{p^*}\int_\Omega|\boldsymbol z|^{p^*}.
\]
Fix \(\boldsymbol u,\boldsymbol\varphi\in W_\perp^{1,p}\), and set
\(
  \boldsymbol v:=\boldsymbol u+\boldsymbol d(\boldsymbol u),
  \ 
  \boldsymbol v_t:=\boldsymbol u+
  \boldsymbol d(\boldsymbol u+t\boldsymbol\varphi).
\)
The minimizing property above yields
\begin{align*}
  G(\boldsymbol u+t\boldsymbol\varphi)-G(\boldsymbol u)
  \leq\mathcal P(\boldsymbol v+t\boldsymbol\varphi)-\mathcal P(\boldsymbol v),\quad
  G(\boldsymbol u+t\boldsymbol\varphi)-G(\boldsymbol u)
  \geq\mathcal P(\boldsymbol v_t+t\boldsymbol\varphi)-\mathcal P(\boldsymbol v_t).
\end{align*}
By the mean value theorem, there exist
\(\tau_1,\tau_2\in(0,1)\), depending on \(t\), such that
\begin{align*}
  \mathcal P(\boldsymbol v+t\boldsymbol\varphi)-\mathcal P(\boldsymbol v)
  =t\mathcal P'(\boldsymbol v+\tau_1t\boldsymbol\varphi)[\boldsymbol\varphi],\quad
  \mathcal P(\boldsymbol v_t+t\boldsymbol\varphi)-\mathcal P(\boldsymbol v_t)
  =t\mathcal P'(\boldsymbol v_t+\tau_2t\boldsymbol\varphi)[\boldsymbol\varphi].
\end{align*}
Lemma~\ref{lem:centering} gives
\(\boldsymbol v_t\to\boldsymbol v\) in \(L^{p^*}\) as \(t\to0\).
Consequently, the preceding two inequalities and the continuity of
\(\mathcal P'\) imply
\[
  \lim_{t\to0}
  \frac{G(\boldsymbol u+t\boldsymbol\varphi)-G(\boldsymbol u)}{t}
  =\mathcal P'(\boldsymbol v)[\boldsymbol\varphi]
  =\int_\Omega
  |\boldsymbol u+\boldsymbol d(\boldsymbol u)|^{p^*-2}
  (\boldsymbol u+\boldsymbol d(\boldsymbol u))\cdot
  \boldsymbol\varphi.
\]
The right-hand side is linear and continuous in \(\boldsymbol\varphi\).
Moreover, it depends continuously on \(\boldsymbol u\), because
Lemma~\ref{lem:centering} and the continuity of the duality map give
\[
  \bigl\|G'(\boldsymbol u_n)-G'(\boldsymbol u)\bigr\|_{(W_\perp^{1,p})'}
  \longrightarrow0
  \qquad\text{whenever }\boldsymbol u_n\to\boldsymbol u
  \text{ in }W_\perp^{1,p}.
\]
Therefore \(G\in C^1\), and \eqref{eq:derivative-centered-functional}
holds.
\end{proof}

\begin{lemma}\label{lem:resonant-energy-splitting}
There exists a bounded minimizing sequence
\((\boldsymbol u_n)\subset
\mathcal H^{-1}(\mathbb S_{p^*,m}\cap\mathcal A_m)\) for the last
problem in \eqref{eq:resonant-sphere-reduction} such that, setting
\(
  \boldsymbol v_n:=\mathcal H(\boldsymbol u_n)
  =\boldsymbol u_n+\boldsymbol d(\boldsymbol u_n),
\)
we have
\[
  \boldsymbol v_n\in\mathbb S_{p^*,m}\cap\mathcal A_m,
  \qquad
  D\boldsymbol u_n\to D\boldsymbol u
  \quad\text{a.e. in }\Omega
\]
for some
\(\boldsymbol u\in W_\perp^{1,p}(\Omega;\mathbb R^m)\).
If \(\boldsymbol w_n:=\boldsymbol u_n-\boldsymbol u\), then
\begin{equation}\label{eq:resonant-brezis-lieb-energy}
  \int_\Omega|D\boldsymbol u_n|^p
  =\int_\Omega|D\boldsymbol u|^p
  +\int_\Omega|D\boldsymbol w_n|^p+o(1).
\end{equation}
\end{lemma}

\begin{proof}
By Lemma~\ref{lem:centered-functional-c1},
\[
  \mathcal M
  :=\mathcal H^{-1}(\mathbb S_{p^*,m}\cap\mathcal A_m)
  =\left\{\boldsymbol u\in W_\perp^{1,p}:G(\boldsymbol u)=\frac1{p^*}\right\}
\]
is closed in \(W_\perp^{1,p}\), and hence complete. It is also a \(C^1\)
manifold, since
\[
  G'(\boldsymbol u)[\boldsymbol u]
  =\int_\Omega
  |\boldsymbol u+\boldsymbol d(\boldsymbol u)|^{p^*}=1
  \qquad\text{for }\boldsymbol u\in\mathcal M.
\]
Applying Ekeland's variational principle to
\(\frac 1p\int_\Omega|D\boldsymbol u|^p\) on \(\mathcal M\), we obtain a
minimizing sequence \((\boldsymbol u_n)\subset\mathcal M\), numbers
\(\lambda_n\in\mathbb R\), and
\(\boldsymbol\sigma_n\to\boldsymbol0\) in \((W_\perp^{1,p})'\) such that,
with \(\boldsymbol v_n=\mathcal H(\boldsymbol u_n)\),
\begin{equation}\label{eq:resonant-approximate-euler-lagrange}
  \int_\Omega
  |D\boldsymbol u_n|^{p-2}D\boldsymbol u_n:D\boldsymbol\varphi
  -\lambda_n\int_\Omega
  |\boldsymbol v_n|^{p^*-2}\boldsymbol v_n\cdot\boldsymbol\varphi
  =\langle\boldsymbol\sigma_n,\boldsymbol\varphi\rangle
\end{equation}
for every \(\boldsymbol\varphi\in W_\perp^{1,p}\). Testing with
\(\boldsymbol u_n\), using \eqref{eq:centering-map} and the minimizing
property, yields \(\lambda_n=S_{0,m}+o(1)\). The resulting scalar error can
be absorbed into \(\boldsymbol\sigma_n\), because
\(\|\boldsymbol v_n\|_{L^{p^*}}=1\). Moreover, subtracting the mean from an
arbitrary test function changes neither term in
\eqref{eq:resonant-approximate-euler-lagrange}, again by
\eqref{eq:centering-map}. Hence
\begin{equation}\label{eq:resonant-euler-lagrange-reduced}
  -\boldsymbol\Delta_p\boldsymbol u_n
  -S_{0,m}|\boldsymbol v_n|^{p^*-2}\boldsymbol v_n
  =\boldsymbol\sigma_n,
  \qquad
  \boldsymbol\sigma_n\to\boldsymbol0
\end{equation}
in \(\bigl(W^{1,p}(\Omega;\mathbb R^m)\bigr)'\).

The sequence is bounded in
\(W_\perp^{1,p}(\Omega;\mathbb R^m)\). Hence, up to a subsequence,
\[
  \boldsymbol u_n\rightharpoonup\boldsymbol u
  \quad\text{in }W^{1,p}(\Omega;\mathbb R^m),
  \qquad
  \boldsymbol u_n\to\boldsymbol u
  \quad\text{in }L^p(\Omega;\mathbb R^m)
\]
and almost everywhere in \(\Omega\). Write
\[
  -\boldsymbol\Delta_p\boldsymbol u_n
  =\boldsymbol f_n+\boldsymbol g_n,
  \qquad
  \boldsymbol f_n:=\boldsymbol\sigma_n,
  \qquad
  \boldsymbol g_n:=
  S_{0,m}|\boldsymbol v_n|^{p^*-2}\boldsymbol v_n.
\]
Then \(\boldsymbol f_n\to\boldsymbol0\) in the dual space, while
\((\boldsymbol g_n)\) is bounded in
\(L^{(p^*)'}(\Omega;\mathbb R^m)\), and hence in
\(L^1(\Omega;\mathbb R^m)\). The almost-everywhere convergence theorem for
nonlinear elliptic systems in \cite{DalMasoMurat1998} therefore yields
\[
  D\boldsymbol u_n\to D\boldsymbol u
  \qquad\text{a.e. in }\Omega.
\]
The splitting identity follows from the Brezis--Lieb lemma
\cite{BrezisLieb1983}.
\end{proof}

We now prove the vectorial counterpart of
Theorem~\ref{thm:bounded-compactness} at resonance.

\begin{theorem}
\label{thm:resonant-compactness}
Assume that
\[
  S_{0,m}<\frac{S}{2^{p/N}}.
\]
Then \(S_{0,m}\) is attained by some
\(\boldsymbol v\in\mathcal A_m\).
\end{theorem}

\begin{proof}
Let \((\boldsymbol u_n)\subset
\mathcal H^{-1}(\mathbb S_{p^*,m}\cap\mathcal A_m)\) be the
minimizing sequence given by Lemma~\ref{lem:resonant-energy-splitting}, set
\[
  \boldsymbol v_n:=\mathcal H(\boldsymbol u_n)
  =\boldsymbol u_n+\boldsymbol d(\boldsymbol u_n),
\]
and let \(\boldsymbol u\in W_\perp^{1,p}\) be the weak limit of
\((\boldsymbol u_n)\). Since \(\|\boldsymbol v_n\|_{L^{p^*}}=1\),
\eqref{eq:centering-homeomorphism-inverse} shows that
\((\boldsymbol d(\boldsymbol u_n))\) is bounded. Passing to a subsequence,
let
\[
  \boldsymbol d(\boldsymbol u_n)\longrightarrow\overline{\boldsymbol d}
  \quad\text{in }\mathbb R^m,
  \qquad
  \boldsymbol v:=\boldsymbol u+\overline{\boldsymbol d}.
\]
Then \(\boldsymbol v_n\to\boldsymbol v\) almost everywhere and in
\(L^p(\Omega;\mathbb R^m)\). Moreover,
\[
  |\boldsymbol v_n|^{p^*-2}\boldsymbol v_n
  \rightharpoonup
  |\boldsymbol v|^{p^*-2}\boldsymbol v
  \quad\text{in }L^{(p^*)'}(\Omega;\mathbb R^m).
\]
Since \(\boldsymbol v_n\in\mathcal A_m\), passing to the limit gives
\begin{equation}\label{eq:resonant-limit-centering}
  \int_\Omega
  |\boldsymbol v|^{p^*-2}\boldsymbol v=\boldsymbol0.
\end{equation}
Thus uniqueness in Lemma~\ref{lem:centering} yields
\(\overline{\boldsymbol d}=\boldsymbol d(\boldsymbol u)\), and hence
\(\boldsymbol v=\mathcal H(\boldsymbol u)\).

Set \(\boldsymbol z_n:=\boldsymbol v_n-\boldsymbol v\). Notice that
\(D\boldsymbol z_n=D(\boldsymbol u_n-\boldsymbol u)\). By the
Brezis--Lieb lemma,
\begin{equation}\label{eq:resonant-brezis-lieb-mass}
  1=\int_\Omega|\boldsymbol v|^{p^*}
  +\int_\Omega|\boldsymbol z_n|^{p^*}+o(1).
\end{equation}

Applying Cherrier's inequality to $|\boldsymbol{z}_n|$, using
\[
|\nabla|\boldsymbol{z}_n||\le |D\boldsymbol{z}_n|,
\]
and recalling that $\boldsymbol{z}_n\to\boldsymbol{0}$
in $L^p(\Omega;\mathbb{R}^m)$, we obtain, for every $\eta>0$,
\begin{equation}\label{eq:resonant-cherrier-defect}
  \left(\int_\Omega|\boldsymbol z_n|^{p^*}\right)^{p/p^*}
  \leq\left(\frac{2^{p/N}}{S}+\eta\right)
  \int_\Omega|D\boldsymbol z_n|^p+o(1).
\end{equation}
Together with \eqref{eq:resonant-brezis-lieb-mass}, this gives
\[
  1\leq\int_\Omega|\boldsymbol v|^{p^*}
  +\left[
  \left(\frac{2^{p/N}}{S}+\eta\right)
  \int_\Omega|D\boldsymbol z_n|^p
  \right]^{p^*/p}+o(1).
\]
Choose \(\eta>0\) so small that
\[
  S_{0,m}\left(\frac{2^{p/N}}{S}+\eta\right)<1.
\]
Multiplying the preceding inequality by \(S_{0,m}^{p^*/p}\), we find
\begin{equation}\label{eq:resonant-energy-comparison}
\begin{aligned}
  S_{0,m}^{p^*/p}
  \leq S_{0,m}^{p^*/p}
  \int_\Omega|\boldsymbol v|^{p^*}+
  \left[
  S_{0,m}\left(\frac{2^{p/N}}{S}+\eta\right)
  \right]^{p^*/p}
  \left(\int_\Omega|D\boldsymbol z_n|^p\right)^{p^*/p}
  +o(1).
\end{aligned}
\end{equation}

On the other hand, the minimizing property and
\eqref{eq:resonant-brezis-lieb-energy} imply
\[
  S_{0,m}+o(1)
  =\int_\Omega|D\boldsymbol u|^p
  +\int_\Omega|D\boldsymbol z_n|^p.
\]
By \eqref{eq:resonant-limit-centering}, the definition of \(S_{0,m}\) gives
\[
  \int_\Omega|D\boldsymbol v|^p
  \geq S_{0,m}
  \left(\int_\Omega|\boldsymbol v|^{p^*}\right)^{p/p^*},
\]
where the inequality is trivial if \(\boldsymbol v=\boldsymbol0\). Since
\(D\boldsymbol v=D\boldsymbol u\), consequently
\[
  S_{0,m}
  \geq S_{0,m}
  \left(\int_\Omega|\boldsymbol v|^{p^*}\right)^{p/p^*}
  +\int_\Omega|D\boldsymbol z_n|^p+o(1).
\]
Raising this inequality to the power \(p^*/p\), we obtain
\[
  S_{0,m}^{p^*/p}
  \geq S_{0,m}^{p^*/p}
  \int_\Omega|\boldsymbol v|^{p^*}
  +\left(\int_\Omega|D\boldsymbol z_n|^p\right)^{p^*/p}
  +o(1).
\]
Combining this estimate with \eqref{eq:resonant-energy-comparison}, we get
\[
  \left\{1-
  \left[
  S_{0,m}\left(\frac{2^{p/N}}{S}+\eta\right)
  \right]^{p^*/p}\right\}
  \left(\int_\Omega|D\boldsymbol z_n|^p\right)^{p^*/p}
  \leq o(1).
\]
The expression in braces is strictly positive, and therefore
\[
  \int_\Omega|D\boldsymbol z_n|^p\longrightarrow0.
\]
It follows from \eqref{eq:resonant-brezis-lieb-mass} and
\eqref{eq:resonant-cherrier-defect} that
\[
  \int_\Omega|\boldsymbol v|^{p^*}=1.
\]
Thus \(\boldsymbol v\in\mathbb S_{p^*,m}\cap\mathcal A_m\), and
\eqref{eq:resonant-brezis-lieb-energy} finally gives
\[
  \int_\Omega|D\boldsymbol v|^p=S_{0,m}.
\]
Hence \(\boldsymbol v\) is a minimizer.
\end{proof}

{
\begin{proposition}\label{prop:scalar-resonant-strict-bound}
Assume that \(N(N-2p+1)>p\). Then
\[
  S_0<\frac{S}{2^{p/N}}.
\]
\end{proposition}
}

{
\begin{proof}
Recall that \(U_\varepsilon=U_{\varepsilon^{(p-1)/p},0}\) is the
Aubin--Talenti bubble used in the proof of
Lemma~\ref{lem:ferrero-strict-bound}, restricted to \(\Omega\) and
concentrating, as \(\varepsilon\to0\), at a boundary point of positive
mean curvature, taken to be the origin.
Let \(d_\varepsilon\in\mathbb R\) be given by
Lemma~\ref{lem:centering}, so that
\begin{equation}\label{eq:centered-bubble-condition}
  \int_\Omega
  |U_\varepsilon+d_\varepsilon|^{p^*-2}
  (U_\varepsilon+d_\varepsilon)=0.
\end{equation}
Since \(U_\varepsilon>0\), we have \(d_\varepsilon<0\). Splitting the
integral in \eqref{eq:centered-bubble-condition} over the positive and
negative sets gives
\begin{equation}\label{eq:centered-bubble-mass-balance}
  \int_{\{U_\varepsilon+d_\varepsilon>0\}}
  (U_\varepsilon+d_\varepsilon)^{p^*-1}
  =\int_{\{U_\varepsilon+d_\varepsilon<0\}}
  |U_\varepsilon+d_\varepsilon|^{p^*-1}.
\end{equation}
Consequently,
\begin{equation}\label{eq:centered-bubble-absolute-mass}
  \int_\Omega|U_\varepsilon+d_\varepsilon|^{p^*-1}
  =2\int_{\{U_\varepsilon+d_\varepsilon>0\}}
  (U_\varepsilon+d_\varepsilon)^{p^*-1}
  \leq2\int_\Omega U_\varepsilon^{p^*-1}.
\end{equation}
Moreover, by the triangle inequality,
\[
  |\Omega|\,|d_\varepsilon|^{p^*-1}
  \leq C\int_\Omega\left(
  U_\varepsilon^{p^*-1}
  +|U_\varepsilon+d_\varepsilon|^{p^*-1}
  \right),
\]
Combining this estimate with
\eqref{eq:centered-bubble-absolute-mass}, we obtain
\begin{equation}\label{eq:centered-bubble-shift-estimate}
  |d_\varepsilon|^{p^*-1}
  \leq C\int_\Omega U_\varepsilon^{p^*-1}.
\end{equation}

Using
\[
  |a-b|^{p^*}
  \geq a^{p^*}-C\left(ba^{p^*-1}+b^{p^*}\right),
  \qquad a,b\geq0,
\]
we obtain the required lower estimate
\begin{align}
  \int_\Omega|U_\varepsilon+d_\varepsilon|^{p^*}
  &\geq\int_\Omega U_\varepsilon^{p^*}
  -C|d_\varepsilon|
  \int_\Omega U_\varepsilon^{p^*-1}
  -C|d_\varepsilon|^{p^*}\notag\\
  &\geq\int_\Omega U_\varepsilon^{p^*}
  -C\left(
  \int_\Omega U_\varepsilon^{p^*-1}
  \right)^{p^*/(p^*-1)}.
  \label{eq:centered-bubble-lower-bound}
\end{align}
The change of variables
\(x=\varepsilon^{(p-1)/p}y\) gives
\[
  \int_\Omega U_\varepsilon^{p^*-1}
  =O\left(
  \varepsilon^{\frac{(N-p)(p-1)}{p^2}}
  \right).
\]
Therefore, by the boundary-bubble estimates in
\cite[Lemmas~1 and~2]{Ferrero2006},
\begin{align*}
  \int_\Omega|U_\varepsilon+d_\varepsilon|^{p^*}
  &\geq\frac{K_2}{2}
  -B_{N,p}H(0)\varepsilon^{(p-1)/p}
  -C\varepsilon^{\gamma_{N,p}}
  +o\bigl(\varepsilon^{(p-1)/p}\bigr),
\end{align*}
where
\[
  \gamma_{N,p}
  =\frac{N(N-p)(p-1)}
  {p\bigl(N(p-1)+p\bigr)}.
\]
By the assumption
\[
  {N(N-2p+1)>p},
\]
we have
\[
  \gamma_{N,p}>\frac{p-1}{p}.
\]
Thus $\varepsilon^{\gamma_{N,p}}=o\left(\varepsilon^{(p-1)/p}\right)$. Since
\(
\nabla(U_\varepsilon+d_\varepsilon)=\nabla U_\varepsilon
\), reasoning as in the proof of Lemma~\ref{lem:ferrero-strict-bound},
we get
\[
  S_0<\frac{S}{2^{p/N}}
\]
for sufficiently small \(\varepsilon\).
This proves the claim.
\end{proof}
}

\begin{proof}[Proof of Theorem~\ref{thm:resonant-existence}]
For every \(w\in\mathcal A_1\) and every
\(\boldsymbol c\in\mathbb S^{m-1}\), the map
\(\boldsymbol c w\) belongs to \(\mathcal A_m\) and has the same Sobolev
quotient as \(w\). Therefore
\[
  S_{0,m}\leq S_0.
\]
Since {\(N(N-2p+1)>p\)},
Proposition~\ref{prop:scalar-resonant-strict-bound}
gives
\[
  S_{0,m}\leq S_0<\frac{S}{2^{p/N}}.
\]
Theorem~\ref{thm:resonant-compactness} then yields a minimizer
\(\boldsymbol u\in\mathcal A_m\), which we normalize so that
\[
  \int_\Omega|\boldsymbol u|^{p^*}=1,
  \qquad
  \int_\Omega|D\boldsymbol u|^p=S_{0,m}.
\]

Set
\[
  \mathcal M
  :=\mathcal H^{-1}(\mathbb S_{p^*,m}\cap\mathcal A_m)
  =\left\{
  \boldsymbol w\in W_\perp^{1,p}(\Omega;\mathbb R^m):
  \int_\Omega
  |\boldsymbol w+\boldsymbol d(\boldsymbol w)|^{p^*}=1
  \right\},
  \qquad
  \widetilde{\boldsymbol u}:=\mathcal H^{-1}(\boldsymbol u)\in\mathcal M.
\]
Since \(\widetilde{\boldsymbol u}\) minimizes
\[
  \boldsymbol w\longmapsto\int_\Omega|D\boldsymbol w|^p
\]
on the \(C^1\) manifold \(\mathcal M\), the Lagrange multiplier theorem
gives some
\(\lambda\in\mathbb R\) such that
\[
  \int_\Omega
  |D{\boldsymbol u}|^{p-2}D{\boldsymbol u}
  :D\boldsymbol\varphi
  =\lambda\int_\Omega
  |\boldsymbol u|^{p^*-2}\boldsymbol u\cdot\boldsymbol\varphi
\]
for every \(\boldsymbol\varphi\in
W_\perp^{1,p}(\Omega;\mathbb R^m)\). Taking
\(\boldsymbol\varphi=\widetilde{\boldsymbol u}\), and using
\eqref{eq:centering-map}, yields \(\lambda=S_{0,m}\). Moreover, for an
arbitrary \(\boldsymbol\varphi\in W^{1,p}(\Omega;\mathbb R^m)\), we may
apply the preceding identity to
\[
  \boldsymbol\varphi
  -\frac1{|\Omega|}\int_\Omega\boldsymbol\varphi.
\]
 We conclude that
\[
  \int_\Omega |D\boldsymbol u|^{p-2}D\boldsymbol u:D\boldsymbol\varphi
  =S_{0,m}\int_\Omega
  |\boldsymbol u|^{p^*-2}\boldsymbol u\cdot\boldsymbol\varphi
\]
for every \(\boldsymbol\varphi\in W^{1,p}(\Omega;\mathbb R^m)\). Hence
\(\boldsymbol u\) satisfies
\[
  \begin{cases}
    -\boldsymbol\Delta_p\boldsymbol u
      =S_{0,m}|\boldsymbol u|^{p^*-2}\boldsymbol u
      &\text{in }\Omega,\\
    |D\boldsymbol u|^{p-2}D\boldsymbol u\,\nu=0
      &\text{on }\partial\Omega.
  \end{cases}
\]
Finally,
\[
  \boldsymbol v:=S_{0,m}^{1/(p^*-p)}\boldsymbol u
\]
is a nontrivial weak solution of \eqref{P0}. We complete the proof
with the regularity argument.

We first adapt Step~1 in the proof of
\cite[Theorem~3.1]{MauroMultiplicityRegularity}. We claim that
\begin{equation}\label{eq:resonant-all-finite-integrability}
  \boldsymbol v\in L^q(\Omega;\mathbb R^m)
  \qquad\text{for every }q>1.
\end{equation}
Fix \(s\geq0\) and \(L>0\), and use the test function
\[
  \boldsymbol\varphi
  :=\boldsymbol v\min\{|\boldsymbol v|^{2s},L^2\}.
\]
Set also
\[
  \boldsymbol\psi
  :=\boldsymbol v
  \left(\min\{|\boldsymbol v|^{2s},L^2\}\right)^{1/p}.
\]
Testing \eqref{P0} directly with
\(\boldsymbol\varphi\), we obtain
\[
  \int_\Omega|D\boldsymbol v|^{p-2}D\boldsymbol v
  :D\boldsymbol\varphi
  =\int_\Omega|\boldsymbol v|^{p^*}
  \min\{|\boldsymbol v|^{2s},L^2\}.
\]
Set
\[
  I:=\int_\Omega|D\boldsymbol v|^{p-2}D\boldsymbol v
  :D\boldsymbol\varphi.
\]
As in Step~1 of the cited proof,
\begin{align*}
  I
  &=\int_\Omega|D\boldsymbol v|^p
  \min\{|\boldsymbol v|^{2s},L^2\}+2s\int_{\{|\boldsymbol v|^s\leq L\}}
  |\boldsymbol v|^{2s-1}|D\boldsymbol v|^{p-2}
  D\boldsymbol v:
  (\boldsymbol v\otimes\nabla|\boldsymbol v|)\\
  &=\int_\Omega|D\boldsymbol v|^p
  \min\{|\boldsymbol v|^{2s},L^2\}+2s\int_{\{|\boldsymbol v|^s\leq L\}}
  |\boldsymbol v|^{2s}|D\boldsymbol v|^{p-2}
  |\nabla|\boldsymbol v||^2\\
  &\geq\int_\Omega|D\boldsymbol v|^p
  \min\{|\boldsymbol v|^{2s},L^2\},
\end{align*}
where we have used
\[
  D\boldsymbol v:(\boldsymbol v\otimes\nabla|\boldsymbol v|)
  =|\boldsymbol v|\,|\nabla|\boldsymbol v||^2.
\]

Moreover,
\[
  D\boldsymbol\psi=
  \begin{cases}
  |\boldsymbol v|^{2s/p}D\boldsymbol v
  +\dfrac{2s}{p}|\boldsymbol v|^{2s/p-1}
  \boldsymbol v\otimes\nabla|\boldsymbol v|
  &\text{if }|\boldsymbol v|^s\leq L,\\[6pt]
  L^{2/p}D\boldsymbol v
  &\text{if }|\boldsymbol v|^s\geq L,
  \end{cases}
\]
and hence
\[
  |D\boldsymbol\psi|^p=
  \begin{cases}
  |\boldsymbol v|^{2s}
  \left(
  |D\boldsymbol v|^2
  +\dfrac{4s}{p}\left(1+\dfrac{s}{p}\right)
  |\nabla|\boldsymbol v||^2
  \right)^{p/2}
  &\text{if }|\boldsymbol v|^s\leq L,\\[8pt]
  L^2|D\boldsymbol v|^p
  &\text{if }|\boldsymbol v|^s\geq L.
  \end{cases}
\]
It follows that
\begin{align*}
  \int_\Omega|D\boldsymbol\psi|^p
  &\leq C_2(s)\int_\Omega|D\boldsymbol v|^p
  \min\{|\boldsymbol v|^{2s},L^2\}\leq C_2(s)I
  =C_2(s)\int_\Omega|\boldsymbol v|^{p^*}
  \min\{|\boldsymbol v|^{2s},L^2\},
\end{align*}
where
\[
  C_2(s)=2\kappa_p
  \left[1+\frac{4s}{p}\left(1+\frac{s}{p}\right)\right]^{p/2},
  \qquad
  \kappa_p=
  \begin{cases}
  2^{p/2-1},&p\geq2,\\
  1,&1<p\leq2.
  \end{cases}
\]
Consequently,
\begin{align*}
  \int_\Omega|D\boldsymbol\psi|^p
  +\int_\Omega|\boldsymbol\psi|^p
  &\leq C_2(s)\int_\Omega|\boldsymbol v|^{p^*}
  \min\{|\boldsymbol v|^{2s},L^2\}
  +\int_\Omega|\boldsymbol\psi|^p=C_2(s)\int_\Omega a(x)|\boldsymbol\psi|^p
  +\int_\Omega|\boldsymbol\psi|^p,
\end{align*}
where
\[
  a(x):=|\boldsymbol v(x)|^{p^*-p}\in L^{N/p}(\Omega),
\]
because \(\boldsymbol v\in L^{p^*}(\Omega;\mathbb R^m)\). By the Sobolev
embedding on \(W^{1,p}(\Omega)\),
\begin{equation}\label{eq:resonant-sobolev-psi}
  C_S\|\boldsymbol\psi\|_{p^*}^p
  \leq C_2(s)\int_\Omega a(x)|\boldsymbol\psi|^p
  +\|\boldsymbol\psi\|_p^p
\end{equation}
for some \(C_S>0\). For every \(K>0\), H\"older's inequality gives
\begin{align*}
  \int_\Omega a(x)|\boldsymbol\psi|^p
  &=\int_{\{a\leq K\}}a(x)|\boldsymbol\psi|^p
  +\int_{\{a>K\}}a(x)|\boldsymbol\psi|^p\\
  &\leq K\|\boldsymbol\psi\|_p^p
  +\left(\int_{\{a>K\}}a(x)^{N/p}\right)^{p/N}
  \|\boldsymbol\psi\|_{p^*}^p\\
  &=K\|\boldsymbol\psi\|_p^p
  +\varepsilon(K)\|\boldsymbol\psi\|_{p^*}^p,
\end{align*}
where \(\varepsilon(K)\to0\) as \(K\to\infty\). Choose \(K=K(s)\)
such that
\[
  C_S-C_2(s)\varepsilon(K)\geq\frac{C_S}{2}.
\]
It follows from \eqref{eq:resonant-sobolev-psi} that
\[
  \frac{C_S}{2}\|\boldsymbol\psi\|_{p^*}^p
  \leq\bigl(C_2(s)K+1\bigr)\|\boldsymbol\psi\|_p^p.
\]
Passing to the limit as \(L\to\infty\), we obtain
\begin{equation}\label{eq:resonant-integrability-step}
  \left(
  \int_\Omega
  |\boldsymbol v|^{(2s+p)p^*/p}
  \right)^{p/p^*}
  \leq C(s)\int_\Omega|\boldsymbol v|^{2s+p}.
\end{equation}
Starting from \(\boldsymbol v\in L^{p^*}\), choose successively
\[
  2s_n+p=q_n,
  \qquad q_0=p^*,
  \qquad q_{n+1}=\frac{p^*}{p}q_n.
\]
Estimate \eqref{eq:resonant-integrability-step} shows inductively that
\(\boldsymbol v\in L^{q_n}\) for every \(n\). Since \(q_n\to\infty\),
\eqref{eq:resonant-all-finite-integrability} follows.


Consequently, setting
\[
    \mathbf f:=|\mathbf v|^{p^*-2}\mathbf v,
\]
we have $\mathbf f\in L^r(\Omega;\mathbb R^m)$ for every $r>1$.
Since $\int_\Omega\mathbf f=0$, let $\boldsymbol\psi$ be the unique
mean-zero solution, componentwise, of
\[
\begin{cases}
-\Delta\boldsymbol\psi=\mathbf f & \text{in }\Omega,\\
\partial_\nu\boldsymbol\psi=0 & \text{on }\partial\Omega,\\
\displaystyle\int_\Omega\boldsymbol\psi=0.
\end{cases}
\]
Fix $r>N$. Standard elliptic regularity gives
$\boldsymbol\psi\in W^{2,r}(\Omega;\mathbb R^m)$, and hence
\[
D\boldsymbol\psi\in C^{0,\alpha}(\overline\Omega)
\]
for some $\alpha>0$. Therefore $(\mathcal P_0)$ can be written as
\[
-\operatorname{div}(|D\mathbf v|^{p-2}D\mathbf v)
=-\operatorname{div}(D\boldsymbol\psi).
\]
Choosing $\beta>0$ sufficiently small, the boundary regularity result
\cite[Theorem~4.6]{BreitCianchiDieningKuusiSchwarzacher2017} yields
\[
|D\mathbf v|^{p-2}D\mathbf v\in C^{0,\beta}(\overline\Omega).
\]
Since $\xi\mapsto|\xi|^{p-2}\xi$ is a homeomorphism, it follows that
\[
\mathbf v\in C^1(\overline\Omega;\mathbb R^m).
\]

 The same regularity holds for the minimizer
\(\boldsymbol u\), and the proof is complete.
\end{proof}

{
\begin{remark}
Under the standing assumptions on $\Omega$, every weak solution
$\boldsymbol v\in W^{1,p}(\Omega;\mathbb R^m)$ of $(\mathcal P_0)$
belongs to $C^1(\overline\Omega;\mathbb R^m)$, for every $1<p<N$.
Indeed, the bootstrap argument in the proof of Theorem~1.4
uses only the weak formulation and the Sobolev embedding,
and yields $\boldsymbol v\in L^q(\Omega;\mathbb R^m)$ for
every finite $q>1$. The subsequent regularity argument
therefore applies unchanged.
\end{remark}
}

\subsection{Radial solutions and symmetry breaking}

\begin{lemma}\label{lem:radial-zero-cauchy-data}
Let \(\Omega\) be a ball or an annulus centered at the origin, and let
\(\boldsymbol v(x)=\boldsymbol V(|x|)\) be a radial solution of
\eqref{P0}. If
\[
  r_0\in[0,R]\quad\text{when }\Omega=B_R,
  \qquad
  r_0\in[R_1,R_2]\quad\text{when }
  \Omega=B_{R_2}\setminus\overline{B}_{R_1},
\]
and
\[
  \boldsymbol V(r_0)=\boldsymbol V'(r_0)=\boldsymbol0,
\]
then \(\boldsymbol v\equiv\boldsymbol0\).
\end{lemma}

\begin{proof}

{
The scalar case follows from
\cite[Theorem~4, case~$(\delta)(i)$]{ReichelWalter1997}.
We give a direct energy proof for the vectorial equation; compare
\cite[proof of Theorem~3.2]{HyndKawohlLindqvist2023}
for the analogous argument in the autonomous eigenvalue problem.

Let \(I=(0,R)\) if \(\Omega=B_R\), and
\(I=(R_1,R_2)\) if
\(\Omega=B_{R_2}\setminus\overline{B}_{R_1}\).
Set \(p'=p/(p-1)\) and
\[
  \boldsymbol Z:=|\boldsymbol V'|^{p-2}\boldsymbol V'.
\]
The radial equation gives \(\boldsymbol Z\in C^1(I;\mathbb R^m)\) and
\[
  \boldsymbol Z'
  +\frac{N-1}{r}\boldsymbol Z
  +|\boldsymbol V|^{p^*-2}\boldsymbol V
  =\boldsymbol0
  \qquad\text{in }I.
\]
Define
\[
  E(r):=\frac1{p'}|\boldsymbol Z(r)|^{p'}
  +\frac1{p^*}|\boldsymbol V(r)|^{p^*}
  =\frac{p-1}{p}|\boldsymbol V'(r)|^p
  +\frac1{p^*}|\boldsymbol V(r)|^{p^*}.
\]
Then \(E\in C(\overline I)\cap C^1(I)\) and \(E\geq0\).
Since \(|\boldsymbol Z|^{p'-2}\boldsymbol Z=\boldsymbol V'\),
the chain rule yields
\[
  E'
  =\boldsymbol V'\cdot
  \left(\boldsymbol Z'+|\boldsymbol V|^{p^*-2}\boldsymbol V\right)
  =-\frac{N-1}{r}|\boldsymbol V'|^p
  \leq0
  \qquad\text{in }I.
\]
This computation remains valid at the zeros of \(\boldsymbol V'\),
because \(\boldsymbol z\mapsto|\boldsymbol z|^{p'}\) is of class \(C^1\).
Moreover, setting \(k=(N-1)p'\), we obtain
\[
  (r^kE(r))'
  =\frac{k}{p^*}r^{k-1}|\boldsymbol V(r)|^{p^*}
  \geq0
  \qquad\text{in }I.
\]
By assumption, \(E(r_0)=0\). If \(r_0>0\), the monotonicity
of \(E\) gives \(E=0\) for every \(r\ge r_0\), while that
of \(r^kE\) gives \(E=0\) for every \(0<r\le r_0\).
Continuity then yields \(E\equiv0\) on \(\overline I\).
If \(r_0=0\), continuity at the origin and the monotonicity of \(E\)
give \(0\leq E(r)\leq E(0)=0\) for every \(r\in\overline I\).
In either case, \(\boldsymbol V\equiv\boldsymbol0\), as claimed.

}
\end{proof}

\begin{proof}[Proof of Theorem~\ref{prop:radial-synchronization}]
Write $\boldsymbol v(x)=\boldsymbol V(|x|)$. Let $I=(0,R)$ if
$\Omega=B_R$ and
$I=(R_1,R_2)$ if $\Omega=B_{R_2}\setminus\overline{B}_{R_1}$. With a
slight abuse of notation, we set $\partial I=\{R\}$ in the first case and
$\partial I=\{R_1,R_2\}$ in the second. Its radial equation and Neumann
boundary condition are
\[
  \begin{cases}
    -\left(r^{N-1}|\boldsymbol V'|^{p-2}(V^i)'\right)'
    =r^{N-1}|\boldsymbol V|^{p^*-2}V^i
      &\text{in }I,\\
    |\boldsymbol V'|^{p-2}(V^i)'=0
      &\text{on }\partial I.
  \end{cases}
  \qquad i=1,\ldots,m.
\]
For every $i,j=1,\ldots,m$, subtracting the equation for $V^j$ multiplied
by $V^i$ from the equation for $V^i$ multiplied by $V^j$ gives
\[
  \frac{d}{dr}\left[
  r^{N-1}|\boldsymbol V'|^{p-2}
  \bigl((V^i)'V^j-(V^j)'V^i\bigr)
  \right]=0.
\]
The Neumann condition shows that this constant is zero. Hence
\(
  (V^i)'V^j=(V^j)'V^i
\)
for all $i,j$. Set $\rho:=|\boldsymbol V|$. On the open set where
$\rho>0$, we have
\[
  \left(\frac{V^i}{\rho}\right)'
  =\frac{1}{\rho^3}\sum_{j=1}^m V^j
  \bigl((V^i)'V^j-V^i(V^j)'\bigr)=0,
  \qquad i=1,\ldots,m.
\]
Therefore, on every connected component $J$ of 
 $\{r\in I:\boldsymbol V(r)\neq\boldsymbol0\}$, there exists
$\boldsymbol c_J\in\mathbb S^{m-1}$ such that
\[
  \frac{\boldsymbol V}{|\boldsymbol V|}=\boldsymbol c_J
  \qquad\text{in }J.
\]
It remains to compare the constants associated with two components of this
open set separated by a zero $r_0$ of $\boldsymbol V$. If
$\boldsymbol V'(r_0)\neq\boldsymbol0$, then
\[
  \boldsymbol V(r)
  =(r-r_0)\boldsymbol V'(r_0)+o(|r-r_0|),
  \qquad
  |\boldsymbol V(r)|
  =|r-r_0|\,|\boldsymbol V'(r_0)|+o(|r-r_0|).
\]
Consequently,
\[
  \frac{\boldsymbol V(r)}{|\boldsymbol V(r)|}
  =\frac{r-r_0}{|r-r_0|}
  \frac{\boldsymbol V'(r_0)}{|\boldsymbol V'(r_0)|}+o(1)
  \qquad\text{as }r\to r_0.
\]

Thus the two constants may differ by a sign.

By Lemma~\ref{lem:radial-zero-cauchy-data}, the alternative
$\boldsymbol V'(r_0)=\boldsymbol0$ would force the solution to be trivial.
Consequently, for a nontrivial solution, the constants $\boldsymbol c_J$
associated with any two connected components coincide up to a sign. Fix one
component $J_0$ and set $\boldsymbol c:=\boldsymbol c_{J_0}$. For every
connected component $J$, there is then a sign
$\sigma_J\in\{-1,1\}$ such that
$\boldsymbol c_J=\sigma_J\boldsymbol c$. Define
\[
  w(r):=\boldsymbol c\cdot\boldsymbol V(r).
\]
On each $J$ we have
\[
  w=\boldsymbol c\cdot\boldsymbol V
  =\sigma_J\boldsymbol c_J\cdot\boldsymbol V
  =\sigma_J|\boldsymbol V|,
  \qquad
  \boldsymbol V
  =|\boldsymbol V|\boldsymbol c_J
  =\boldsymbol c\,\sigma_J|\boldsymbol V|
  =\boldsymbol c\,w,
\]
and the same identity holds at the zeros of $\boldsymbol V$. Hence
$\boldsymbol v=\boldsymbol c\,w$ throughout $\Omega$. Moreover,
$w\in C^1(\overline\Omega)$ because $w=\boldsymbol c\cdot\boldsymbol v$,
and substitution into \eqref{P0}, using $|\boldsymbol c|=1$, shows that
$w$ is a scalar radial solution. This proves the assertion.
\end{proof}

We now prove the symmetry-breaking result on the ball.

\begin{proof}[Proof of Corollary~\ref{thm:ball-symmetry-breaking}]
Suppose by contradiction that $\boldsymbol v$ is a nontrivial radial
solution of \eqref{P0}. By
Theorem~\ref{prop:radial-synchronization}, it has the form
$\boldsymbol v=\boldsymbol c\,w$, where
$\boldsymbol c\in\mathbb S^{m-1}$ and $w$ is a nontrivial scalar radial
solution. The compatibility condition becomes
\[
  \int_{B_R}|w|^{p^*-2}w=0,
\]
so $w$ changes sign. Since $w$ is a \(C^1\) radial function on the ball,
symmetry gives \(w'(0)=0\). If \(w(0)=0\), then
Lemma~\ref{lem:radial-zero-cauchy-data}, applied to
\(\boldsymbol V=\boldsymbol c\,w\), would imply
\(\boldsymbol v\equiv\boldsymbol0\), a contradiction. Hence \(w(0)\neq0\).
After changing its sign if necessary, we may assume that \(w(0)>0\), and
its first zero \(r_0\) therefore satisfies \(r_0\in(0,R)\). The restriction
of \(w\) to \(B_{r_0}\) is then a positive solution of
\[
  -\Delta_pw=w^{p^*-1}\quad\text{in }B_{r_0},
  \qquad w=0\quad\text{on }\partial B_{r_0}.
\]
This is impossible by the Pohozaev identity for the critical
$p$-Laplacian on a star-shaped domain \cite{GueddaVeron1989}. Consequently,
no nontrivial solution of \eqref{P0} can be radial. This is the
$p$-Laplacian counterpart of the symmetry-breaking argument in
\cite[Remark~5.1]{SaldanaTavares2022}.
\end{proof}


\begin{thebibliography}{99}

\bibitem{ArioliGazzola1998}
G.~Arioli and F.~Gazzola,
\emph{Some results on $p$-Laplace equations with a critical growth term},
Differential Integral Equations \textbf{11} (1998), no.~2, 311--326.

\bibitem{AdimurthiMancini1991}
Adimurthi and G.~Mancini,
\emph{The Neumann problem for elliptic equations with critical
nonlinearity},
in \emph{Nonlinear Analysis}, Scuola Norm. Sup. Pisa Quaderni,
Scuola Norm. Sup., Pisa, 1991, 9--25.

\bibitem{AdimurthiYadava1994}
Adimurthi and S.~L. Yadava,
\emph{Some remarks on Sobolev type inequalities},
Calc. Var. Partial Differential Equations \textbf{2} (1994), 427--442.

\bibitem{AdimurthiYadavaLinNi1993}
Adimurthi and S.~L. Yadava,
\emph{On a conjecture of Lin--Ni for a semilinear Neumann problem},
Trans. Amer. Math. Soc. \textbf{336} (1993), no.~2, 631--637.

\bibitem{BiagiEspositoRoncoroniVecchi2025}
S.~Biagi, F.~Esposito, A.~Roncoroni, and E.~Vecchi,
\emph{Brezis--Nirenberg type results for the anisotropic
$p$-Laplacian},
J. Lond. Math. Soc. \textbf{112} (2025), e70331.


\bibitem{BreitCianchiDieningKuusiSchwarzacher2017}
D.~Breit, A.~Cianchi, L.~Diening, T.~Kuusi and S.~Schwarzacher,
\emph{The $p$-Laplace system with right-hand side in divergence form:
inner and up to the boundary pointwise estimates},
Nonlinear Anal. \textbf{153} (2017), 200--212.

\bibitem{BrezisLieb1983}
H.~Brezis and E.~Lieb,
\emph{A relation between pointwise convergence of functions and convergence
of functionals},
Proc. Amer. Math. Soc. \textbf{88} (1983), no.~3, 486--490.

\bibitem{CaffarelliGidasSpruck1989}
L.~Caffarelli, B.~Gidas, and J.~Spruck,
\emph{Asymptotic symmetry and local behavior of semilinear elliptic
equations with critical Sobolev growth},
Comm. Pure Appl. Math. \textbf{42} (1989), no.~3, 271--297.

\bibitem{CaninoMauroVectorial}
A.~Canino and S.~Mauro,
\emph{Existence results for variational quasilinear elliptic systems
involving the vectorial $p$-Laplacian},
arXiv preprint arXiv:2510.15694, 2025.

\bibitem{CaninoMauroCompetitive}
A.~Canino and S.~Mauro,
\emph{Quasilinear elliptic cooperative and competitive systems},
arXiv preprint arXiv:2510.18758, 2025.


\bibitem{ChenLi1991}
W.~X. Chen and C.~Li,
\emph{Classification of solutions of some nonlinear elliptic equations},
Duke Math. J. \textbf{62} (1991), no.~3, 615--622.

\bibitem{Cherrier1984}
P.~Cherrier,
\emph{Meilleures constantes dans des in\'egalit\'es relatives aux espaces de
Sobolev},
Bull. Sci. Math. \textbf{108} (1984), no.~2, 225--262.


\bibitem{CHS2026}
H.~Chtioui, H.~Hajaiej and L.~Sharma,
\emph{The role of the mean curvature in nonlinear $p$-Laplacian
problems with critical exponent},
preprint (2026), arXiv:2604.03378.

\bibitem{CianchiMazya2014}
A.~Cianchi and V.~G. Maz'ya,
\emph{Global boundedness of the gradient for a class of nonlinear elliptic
systems},
Arch. Ration. Mech. Anal. \textbf{212} (2014), no.~1, 129--177.

\bibitem{ClappSzulkin2019}
M.~Clapp and A.~Szulkin,
\emph{A simple variational approach to weakly coupled competitive elliptic
systems},
NoDEA Nonlinear Differential Equations Appl. \textbf{26} (2019),
Paper No.~26.

\bibitem{ComteKnaap1991}
M.~Comte and M.~C. Knaap,
\emph{Existence of solutions of elliptic equations involving critical
Sobolev exponents with Neumann boundary condition in general domains},
Differential Integral Equations \textbf{4} (1991), no.~6, 1133--1146.

\bibitem{DalMasoMurat1998}
G.~Dal Maso and F.~Murat,
\emph{Almost everywhere convergence of gradients of solutions to nonlinear
elliptic systems},
Nonlinear Anal. \textbf{31} (1998), no.~3, 405--412.

\bibitem{DamascelliMerchanMontoroSciunzi2014}
L.~Damascelli, S.~Merch\'an, L.~Montoro, and B.~Sciunzi,
\emph{Radial symmetry and applications for a problem involving the
$-\Delta_p(\cdot)$ operator and critical nonlinearity in $\mathbb R^N$},
Adv. Math. \textbf{256} (2014), 313--335.

\bibitem{Ekeland1974}
I.~Ekeland,
\emph{On the variational principle},
J. Math. Anal. Appl. \textbf{47} (1974), 324--353.

\bibitem{Ferrero2006}
A.~Ferrero,
\emph{Least energy solutions for critical growth equations with a lower
order perturbation},
Adv. Differential Equations \textbf{11} (2006), no.~10, 1167--1200.

\bibitem{GarciaAzoreroPeral1987}
J.~P. Garc\'ia Azorero and I.~Peral Alonso,
\emph{Existence and nonuniqueness for the $p$-Laplacian: nonlinear
eigenvalues},
Comm. Partial Differential Equations \textbf{12} (1987), no.~12,
1389--1430.

\bibitem{GueddaVeron1989}
M.~Guedda and L.~V\'eron,
\emph{Quasilinear elliptic equations involving critical Sobolev exponents},
Nonlinear Anal. \textbf{13} (1989), no.~8, 879--902.


\bibitem{HyndKawohlLindqvist2023}
R.~Hynd, B.~Kawohl and P.~Lindqvist,
\emph{On the uniqueness of eigenfunctions for the vectorial $p$-Laplacian},
Arch. Math. (Basel) \textbf{121} (2023), 745--755.

\bibitem{Lieberman1988}
G.~M. Lieberman,
\emph{Boundary regularity for solutions of degenerate elliptic equations},
Nonlinear Anal. \textbf{12} (1988), no.~11, 1203--1219.

\bibitem{MauroMultiplicityRegularity}
S.~Mauro,
\emph{Multiplicity and regularity results for quasilinear elliptic systems
via nonsmooth critical point theory},
arXiv preprint arXiv:2603.24087, 2026.

\bibitem{McCurdySaldanaSchiera2026}
S.~McCurdy, A.~Salda\~na, and D.~Schiera,
\emph{On Neumann $p$-Laplacian Lane--Emden equations and their asymptotic
relationship with relative isoperimetric problems},
arXiv preprint arXiv:2606.08222, 2026.

\bibitem{PariniWeth2015}
E.~Parini and T.~Weth,
\emph{Existence, unique continuation and symmetry of least energy nodal
solutions to sublinear Neumann problems},
Math. Z. \textbf{280} (2015), no.~3--4, 707--732.

\bibitem{PucciServadei2008}
P.~Pucci and R.~Servadei,
\emph{Regularity of weak solutions of homogeneous or inhomogeneous
quasilinear elliptic equations},
Indiana Univ. Math. J. \textbf{57} (2008), no.~7, 3329--3363.

\bibitem{ReichelWalter1997}
W.~Reichel and W.~Walter,
\emph{Radial solutions of equations and inequalities involving the
$p$-Laplacian},
J. Inequal. Appl. \textbf{1} (1997), no.~1, 47--71.

\bibitem{SaldanaTavares2022}
A.~Salda\~na and H.~Tavares,
\emph{On the least-energy solutions of the pure Neumann Lane--Emden
equation},
NoDEA Nonlinear Differential Equations Appl. \textbf{29} (2022), no.~3,
Paper No.~30.

\bibitem{Sciunzi2016}
B.~Sciunzi,
\emph{Classification of positive $\mathcal D^{1,p}(\mathbb R^N)$-solutions
to the critical $p$-Laplace equation in $\mathbb R^N$},
Adv. Math. \textbf{291} (2016), 12--23.

\bibitem{SzulkinWeth2010}
A.~Szulkin and T.~Weth,
\emph{The method of Nehari manifold},
in \emph{Handbook of Nonconvex Analysis and Applications},
Int. Press, Somerville, MA, 2010, 597--632.

\bibitem{Trudinger1967}
N.~S. Trudinger,
\emph{On Harnack-type inequalities and their application to quasilinear
elliptic equations},
Comm. Pure Appl. Math. \textbf{20} (1967), no.~4, 721--747.

\bibitem{Vetois2016}
J.~V\'etois,
\emph{A priori estimates and application to the symmetry of solutions for
critical $p$-Laplace equations},
J. Differential Equations \textbf{260} (2016), no.~1, 149--161.

\end{thebibliography}
\end{document}